\documentclass[final,3p, review,11pt]{elsarticle}

\usepackage{amssymb}
\usepackage{natbib}
\usepackage{graphicx}
\usepackage[usenames,dvipsnames]{color}
\usepackage{amsmath}
\usepackage{bm}
\usepackage{multirow}
\usepackage{hyperref}
\newcommand{\Bezier}{B\'{e}zier\ }
\newcommand{\D}{\displaystyle}
\newcommand{\T}{\textstyle}
\newlength\figwidth
\usepackage{amsthm}

\renewcommand\figurename{Fig. }
\makeatletter
\renewcommand{\fnum@figure}[1]{\textbf{\figurename\thefigure.} }
\makeatother

\hypersetup{colorlinks=true}

\newtheorem{lemma}{Lemma}[section]

\newtheorem{proposition}[lemma]{Proposition}

\def\pb{\textbf{\emph{P}}}
\def\ub{\mathbf{u}}

\def\Pc{\mathcal{P}}
\def\Sc{\mathcal{S}}

\usepackage[switch]{lineno} 

\journal{Computer-Aided Design}

\begin{document}


\begin{frontmatter}

\title{Isogeometric computation reuse method for complex objects with
topology-consistent volumetric parameterization}
 \author[A]{Gang Xu \corref{cor}}\ead{gxu@hdu.edu.cn}
 \author[B,C]{Tsz-Ho Kwok}  
 \author[B,D]{Charlie C.L. Wang }\ead{c.c.wang@tudelft.nl}
 
 \cortext[cor]{Corresponding author.}
 \address[A]{School of Computer Science and Technology, Hangzhou Dianzi University, Hangzhou 310018, P.R.China}
 \address[B]{Department of Mechanical and Automation Engineering, 
The Chinese University of Hong Kong, Hong Kong} 
\address[C]{Department of Mechanical and Industrial Engineering,
Concordia University, Montreal, Quebec, Canada} 
 \address[D]{Department of Design Engineering, 
Delft University of Technology, The Netherlands} 
 
\begin{abstract}
Volumetric spline parameterization and computational efficiency are
two main challenges in \textit{isogeometric analysis} (IGA). To tackle
this problem, we propose a framework of \textit{computation reuse} in
IGA on a set of three-dimensional models with similar semantic
features. Given a template domain, B-spline based consistent
volumetric parameterization is first constructed for a set of models
with similar semantic features. An efficient quadrature-free method is investigated in
our framework to compute the entries of stiffness matrix by \Bezier extraction and polynomial approximation. In our approach, evaluation on the stiffness matrix and imposition of the boundary conditions can be pre-computed and reused during IGA on a set of CAD models. Examples with complex geometry are presented to show the effectiveness of our methods, and efficiency similar to the computation in linear finite element analysis can be achieved for IGA taken on a set of models.
\end{abstract}


\end{frontmatter}

\section{Introduction}
The \textit{isogeometric analysis} (IGA) approach, which was proposed
by Hughes et al. \cite{hughes:CMAME 2005}, offers the possibility of
seamless integration of computational analysis and geometric
design. Two major challenges in the current development of IGA are
volumetric parameterization and computational efficiency. In the
recent book of Cottrell and Hughes~\cite{Cottrell:igabook}, it has
been pointed out that the most significant challenges towards
isogeometric analysis is how to construct analysis-suitable
parameterizations from given CAD models. On the other hand, high-order
basis functions are often employed to achieve smooth solution fields
with high continuity, which however also increases the computational
cost when filling stiffness matrices. In this paper, we propose a
method for \textit{computation reuse} in IGA on three-dimensional
models with similar semantic features, by which the computational
efficiency can be significantly improved.  Applications in
computational design that can be benefit from this research include: 
1) the physical analysis on a family of products having the same
topology but different shapes; 2) using as the inner loop of
physics-based shape optimization, where the computation can be greatly 
speeded up after applying a complete IGA in the first iteration. 

\begin{figure}[t]
\centering
\begin{minipage}[t]{0.24\textwidth}
\centering
\includegraphics[width=\textwidth]{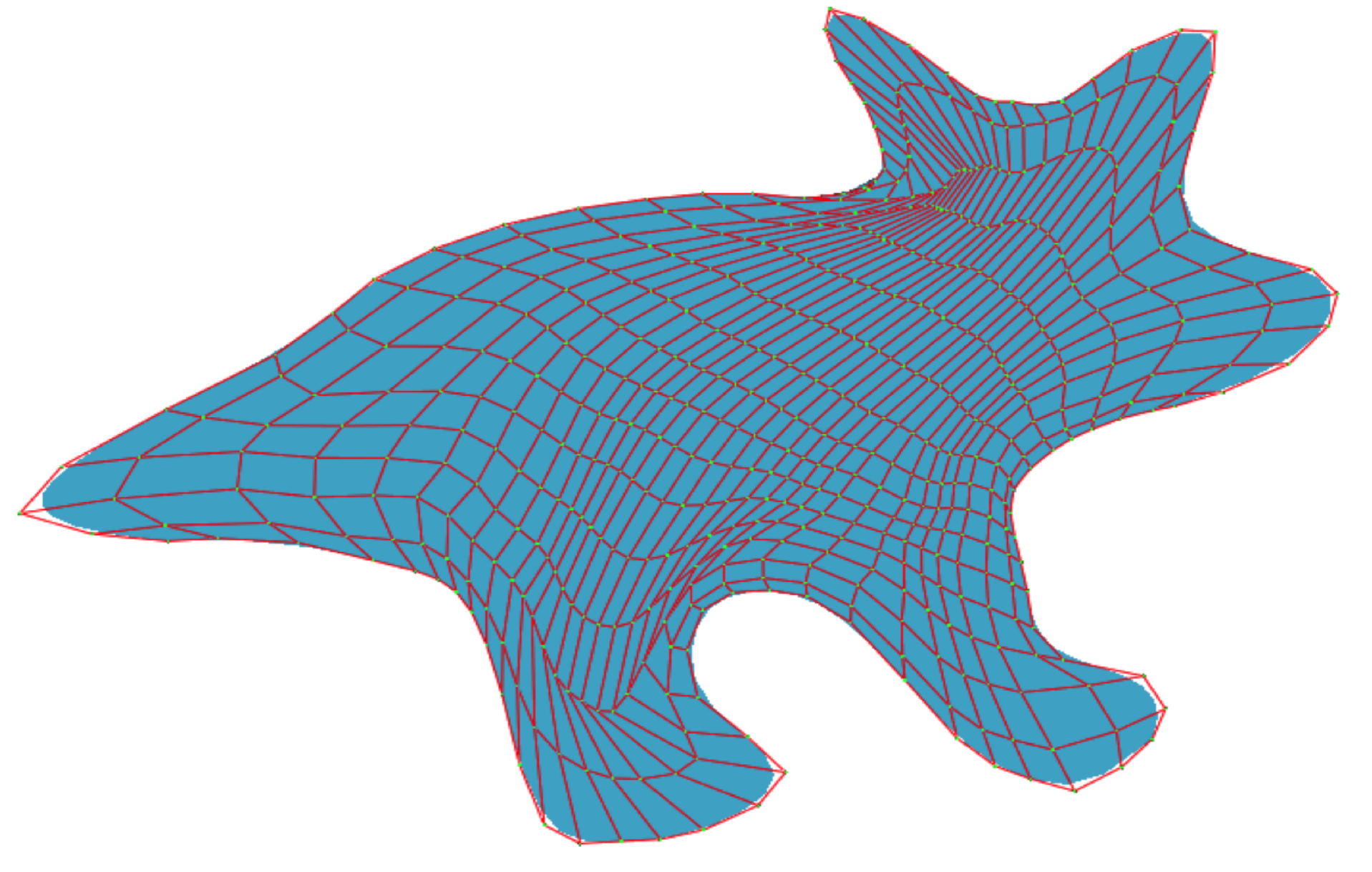}
\\ (a) model A
\end{minipage}
\begin{minipage}[t]{0.24\textwidth}
\centering
\includegraphics[width=.9\textwidth]{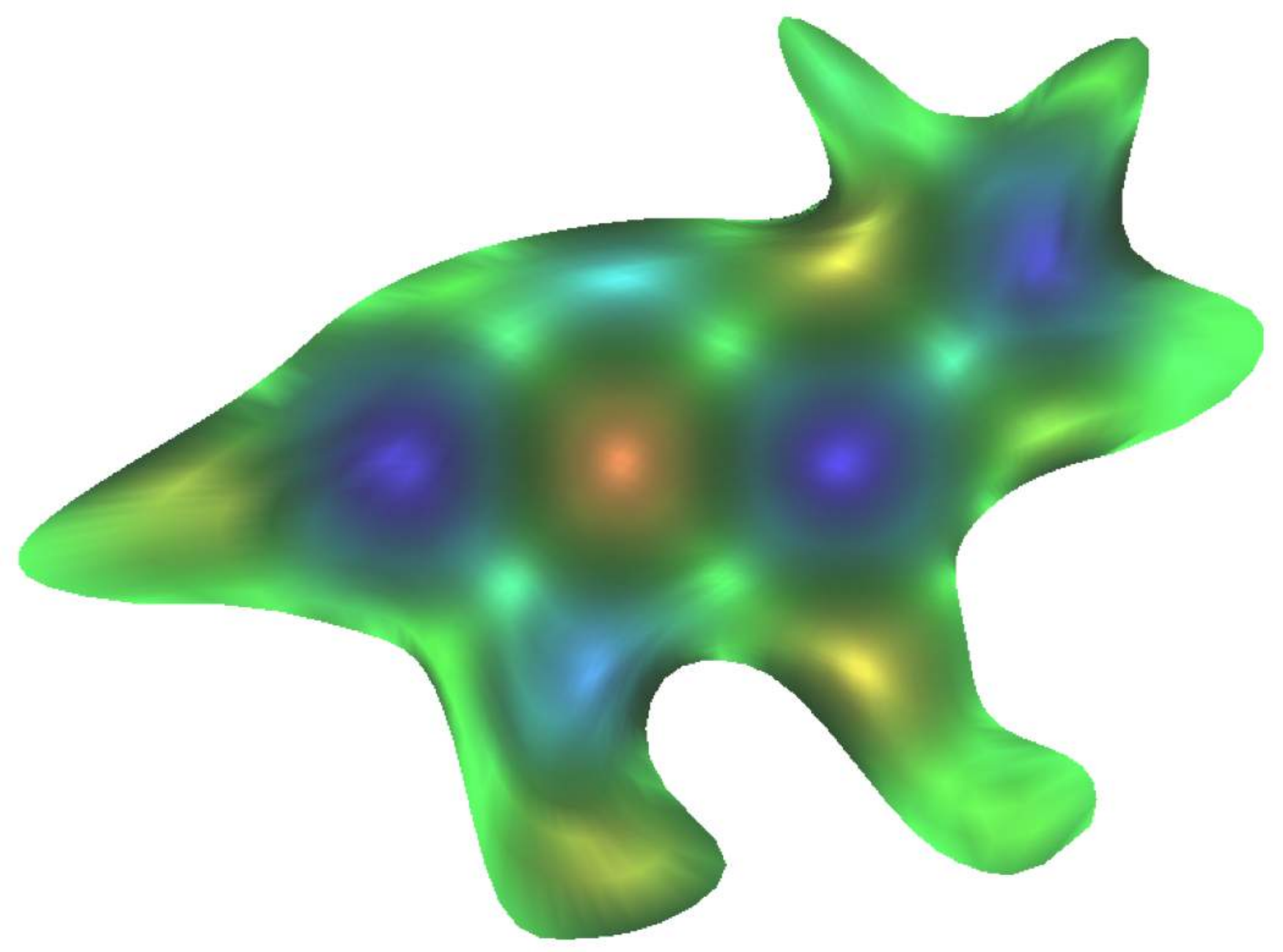}
\\ (b)IGA result of (a)
\end{minipage}
\begin{minipage}[t]{0.24\textwidth}
\centering
\includegraphics[width=.72\textwidth]{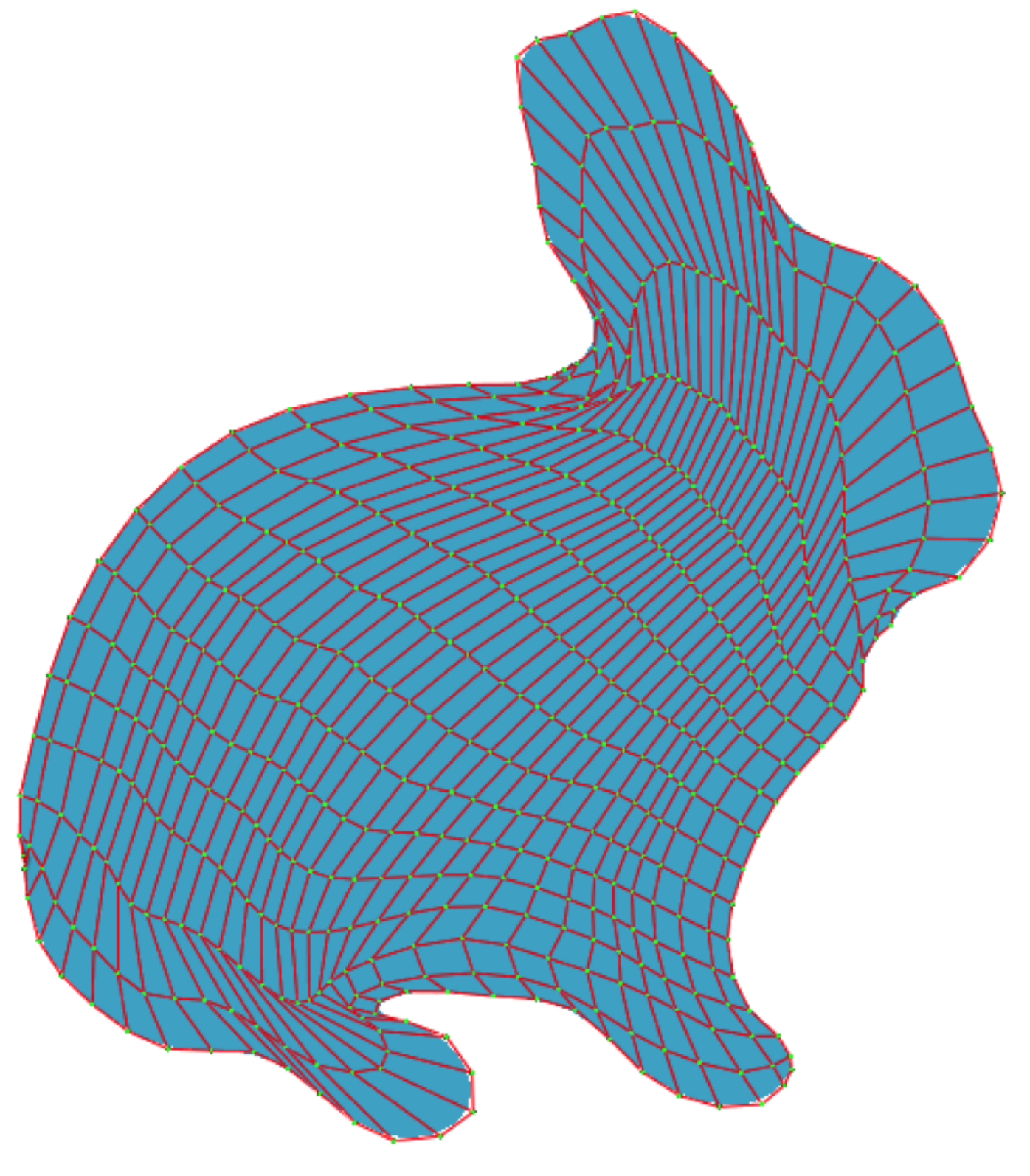}
\\ (c) model B
\end{minipage}
\begin{minipage}[t]{0.24\textwidth}
\centering
\includegraphics[width=\textwidth]{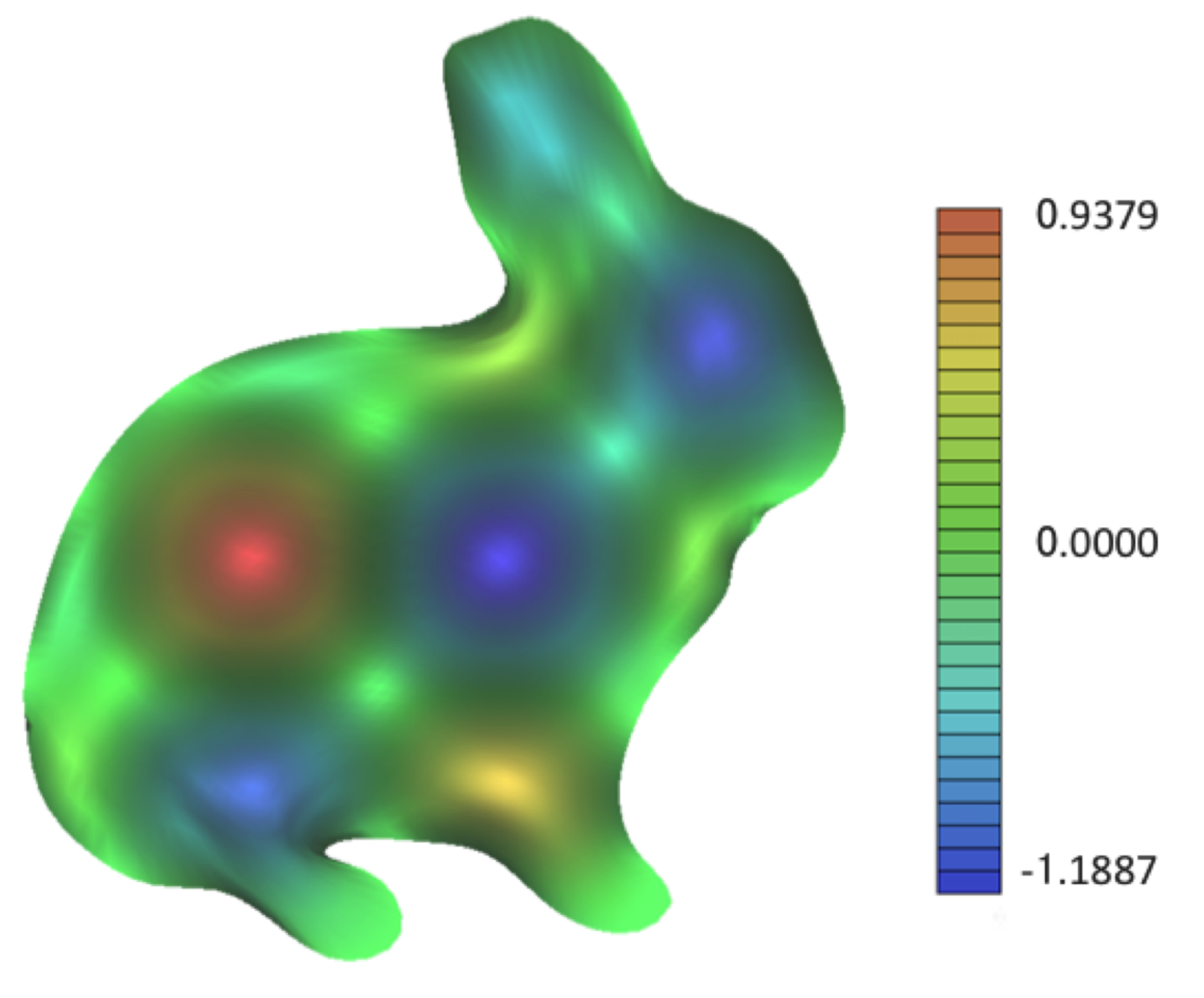}
\\ (d) IGA solution of (c)
\end{minipage}\\
\begin{minipage}[t]{1.5in}
\centering
\includegraphics[width=1.1in]{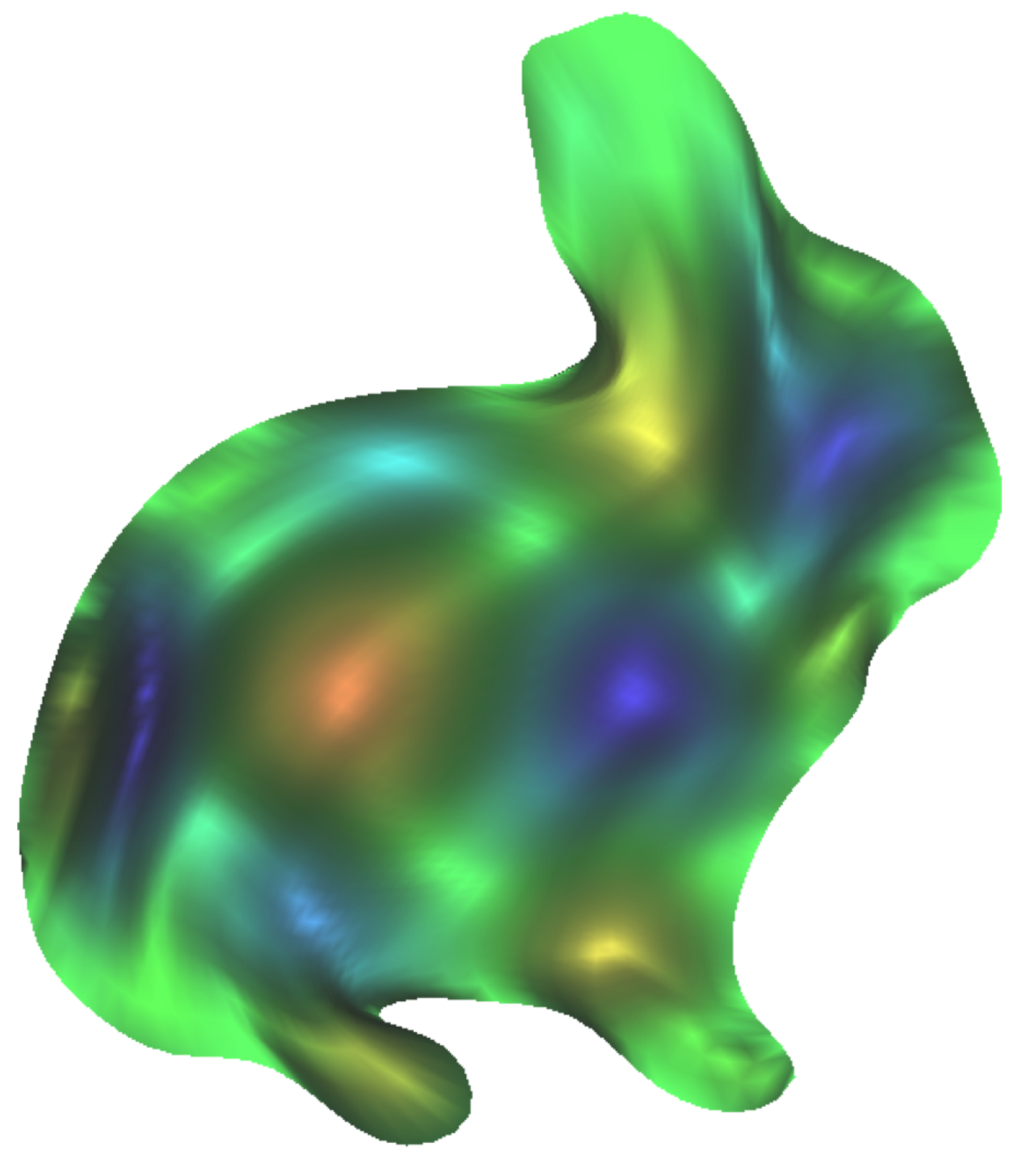}
\\ (e) direct mapping result from (b)
\end{minipage}
\begin{minipage}[t]{1.6in}
\centering
\includegraphics[width=1.6in]{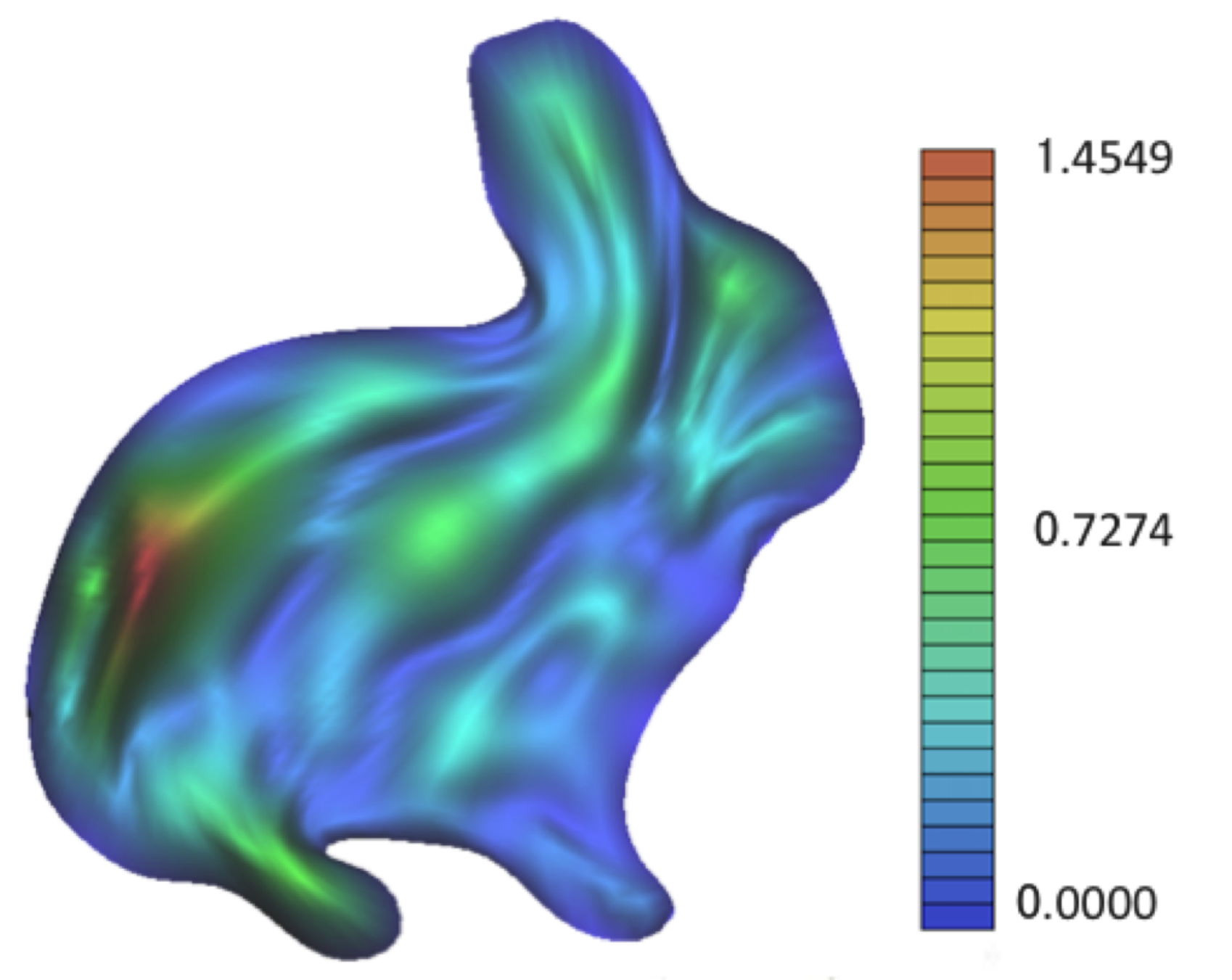}
\\ (f) the difference between (e)and (d) 
\end{minipage}
\begin{minipage}[t]{0.24\textwidth}
\centering
\includegraphics[width=1.05\textwidth]{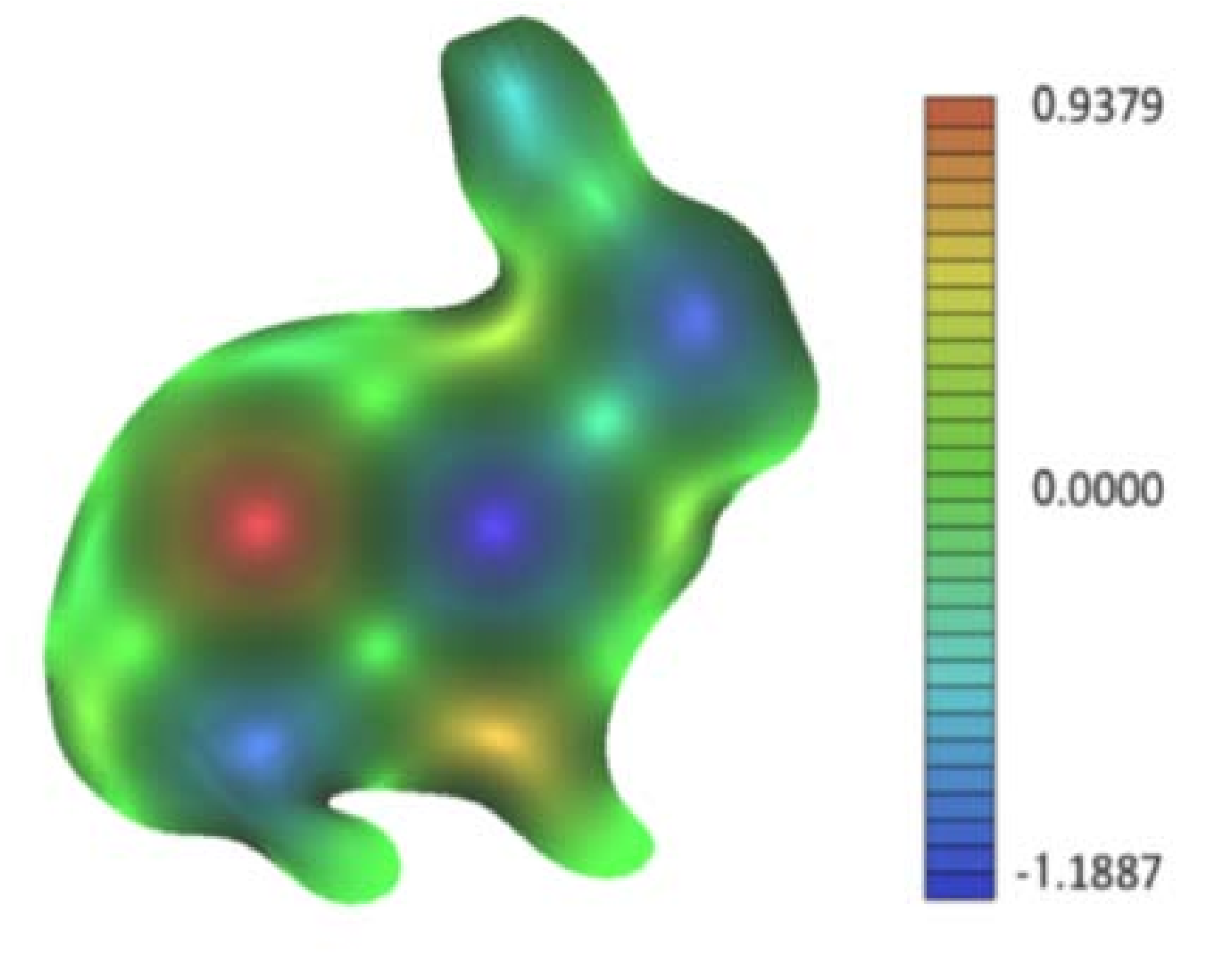}
\\ (g) solution w.r.t the proposed method   
\end{minipage}
\begin{minipage}[t]{1.5in}
\centering
\includegraphics[width=1.5in]{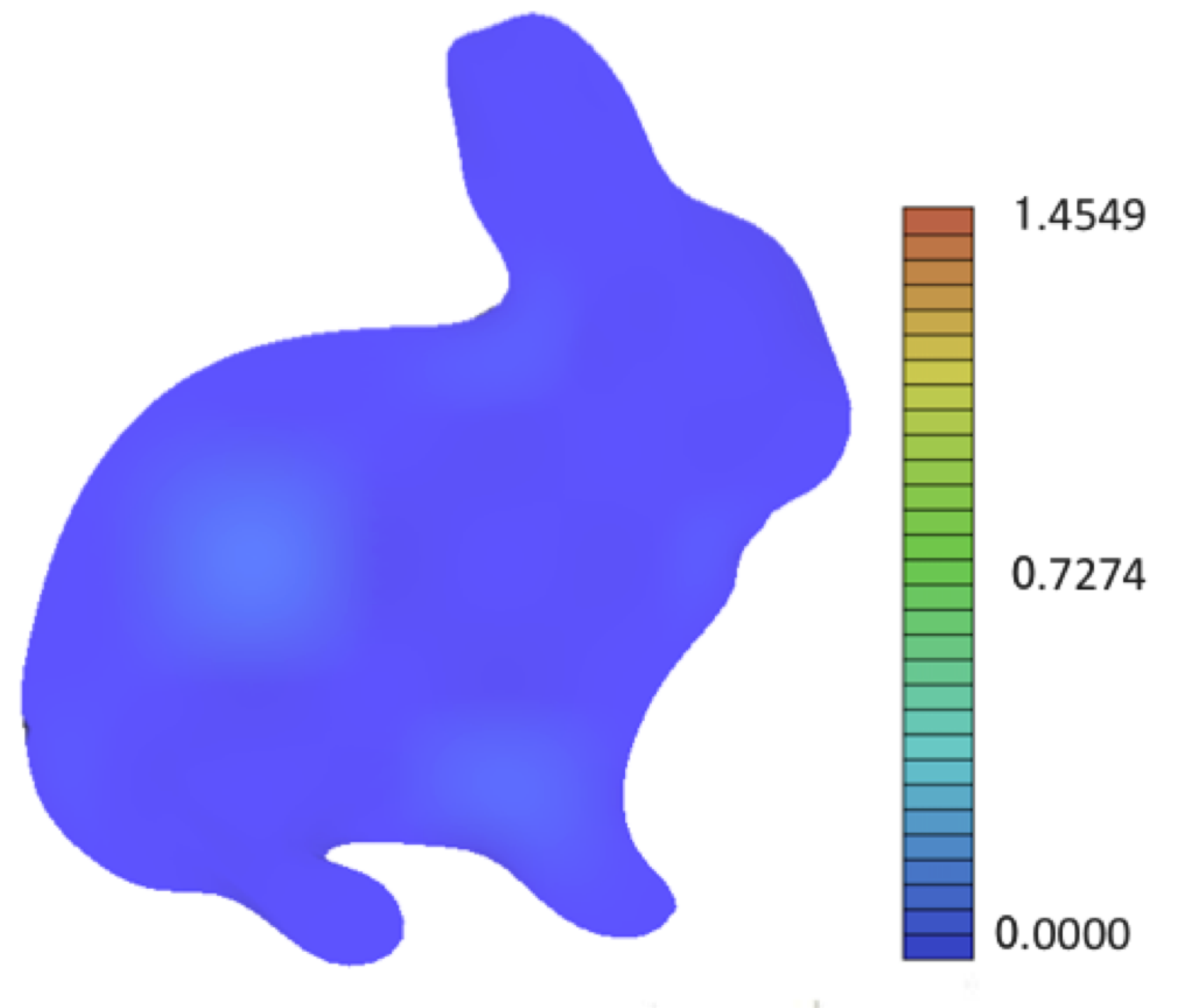}
\\ (h) the difference between (g) and (d) 
\end{minipage}\\
\caption{Comparison of results obtained from IGA computation reuse and
  direct mapping. (a) and (c) are two different planar domains that
  have the same tensor-product B-spline space. (b) and (d) are the IGA
  solution of a heat conduction problem solved on these two domains
  separately. (e) is the result by directly mapping (b) into (c), and
  (f) shows the difference between the solutions  in (e) and (d).  (g)
  is the result obtained by our proposed computation-reuse method, and
  (h) shows the difference  between the solutions in (g) and (d) with the same scale as in (f) -- it is much more accurate.
}
\label{fig:rabbit}
\end{figure}

It should be mentioned that the result of IGA-solution cannot be
directly reused by the mapping of a consistent parameterization
(e.g.,~\cite{KZW12,KZW:TVCG2012, sheffer:sig2004, Hoppe:sig2004}) between models having similar semantic features As illustrated in Fig.\ref{fig:rabbit}(a,b) and (c,d), a heat conduction problem with homogeneous boundary condition and the following source function
\[
f(x,y)= -2\pi^2\sin(\pi x)\sin(\pi y) 
\]
is solved separately over two different planar domains that have the
same tensor-product B-spline space. When directly mapping the result
from Fig.\ref{fig:rabbit}(b) into the domain of (c), the result of
heat distribution is as shown in (e) which is quite different from IGA
result (i.e., Fig.\ref{fig:rabbit}(d)). Fig.\ref{fig:rabbit}(f) shows
the difference in color between (e) and (d). On the other aspects,
when using the computation reuse approach developed in this paper, the 
corresponding solution is presented in  Fig.\ref{fig:rabbit}(g), where
approximation errors brought in are trivial -- see the difference in
color  shown in Fig.\ref{fig:rabbit}(h).

The main contribution of our work can be summarized as follows:
\begin{itemize}
\item An efficient quadrature-free method is proposed to compute the entries of stiffness matrix
  with the help of \Bezier extraction and polynomial approximation techniques applying to trivariate rational \Bezier functions.  
  
\item  We present a framework of computational reuse in IGA and the method for
  reuse when imposing boundary conditions in this framework. Up to
  $12.6$ times speedup can be observed by using our method on problems 
  with large number of degree of freedom (DOF). 
\end{itemize}

\vspace{5pt} \noindent The rest of our paper is organized as follows. In Section \ref{sec:related},  the related work on volumetric parameterization and computational efficiency of IGA will be reviewed.  The method to construct consistent B-spline based volumetric parameterization is presented in Section \ref{sec:volcon}. Section \ref{sec:analysis} presents the quadrature-free IGA
method using the \Bezier extraction and polynomial approximation
techniques. By combining techniques presented in Sections \ref{sec:volcon} 
and \ref{sec:analysis}, the computation reuse framework for a set of models with similar semantic features is presented in Section \ref{sec:example}.  
Lastly, we conclude this paper and discuss possible future works in Section \ref{sec:conclude}.

\section{Related work}
\label{sec:related}
\noindent \textbf{Volumetric parameterization} \hspace{8pt} From the viewpoint of graphics applications, volume parameterization
of 3D models has been studied in \cite{li:spm07,xia:smi10,xia:gmp10}. On the other aspect, there are also some recent work on volume parameterization in the literature of IGA. Martin et al. \cite{martin:CAGD2009}  proposed a method to fit a genus-$0$ triangular mesh by B-spline volume parameterization, based on discrete volumetric harmonic functions.
A variational approach for constructing  NURBS parameterization of
swept volumes is proposed by  Aigner et al. \cite{Aigner:swept}. Escobar et al. \cite{Escobara:CMAME2011} proposed a method to construct a
trivariate T-spline volume of complex genus-zero solids by using an adaptive tetrahedral meshing and
mesh untangling technique.  Zhang et al. \cite{Zhang:cmame2012} proposed a robust and efficient algorithm to
construct injective solid T-splines for genus-zero geometry
from a boundary triangulation. Based on the
Morse theory, a volumetric parameterization method of mesh model with
arbitrary topology is proposed in \cite{Wang:spm12}.   In \cite{xu:cmame,xu:spm2012}, a constraint optimization
framework is proposed to obtain analysis-suitable planar and volumetric
parameterization of computational domain. Pettersen and Skytt proposed the spline volume faring method
to obtain high-quality volume parameterization for isogeometric
applications \cite{Pettersen:2012splinefaring}. Zhang et al. \cite{Zhang:cm2013} studied
the construction of conformal solid T-spline from boundary T-spline
representation by Octree structure and boundary offset.  
 For volume parameterization of
three-manifold solid models having homeomorphic topology, 
Kwok and Wang \cite{Kwok:CAG2013} proposed an algorithm to
constructing volumetric domains with consistent topology. The generated volumetric parameterizations share
the same set of base domains and are constrained by the corresponding semantic features in the form of
anchor points. In this paper, a compactly supported radial basis
function method is proposed to construct consistent volumetric
B-spline parameterization for models with similar semantic features.

\vspace{8pt} \noindent \textbf{Efficiency issues of isogeometric analysis} \hspace{8pt} 
High-order basis functions are used to represent the geometry and the 
physical field in IGA to achieve high-accuracy simulation results. Hence, computational efficiency is a key
issue in the field of isogeometric analysis. In order to improve the efficiency, several methods
have been proposed. There is a trend to use \textit{graphic possessing units} (GPU) to improve the computational efficiency of assembling
the stiffness matrix (e.g., \cite{Karatarakis:cmame2014}). On the other hand, efficiency
improvement on integral computing has also been studied.  Hughes et al. \cite{hughes:quadrature} proposed 
an efficient quadrature rules for NURBS-based isogeometric analysis.  Antolin et al.~\cite{Antolin:cmame2015}  developed a sum-factorization approach to save the quadrature computational cost significantly
based on the tensor-product structure of splines and NURBS shape
functions.  Mantzaflaris and J\"{u}ttler \cite{Mantzaf:cmame2015} presented a quadrature-free integration method by interpolation and look-up table for Galerkin-based isogeometric analysis. In this paper, we propose the concept of
\textit{computation reuse} to improve the efficiency of IGA on a set of CAD models with consistent topology.

\begin{figure}
\centering
\includegraphics[width=\textwidth]{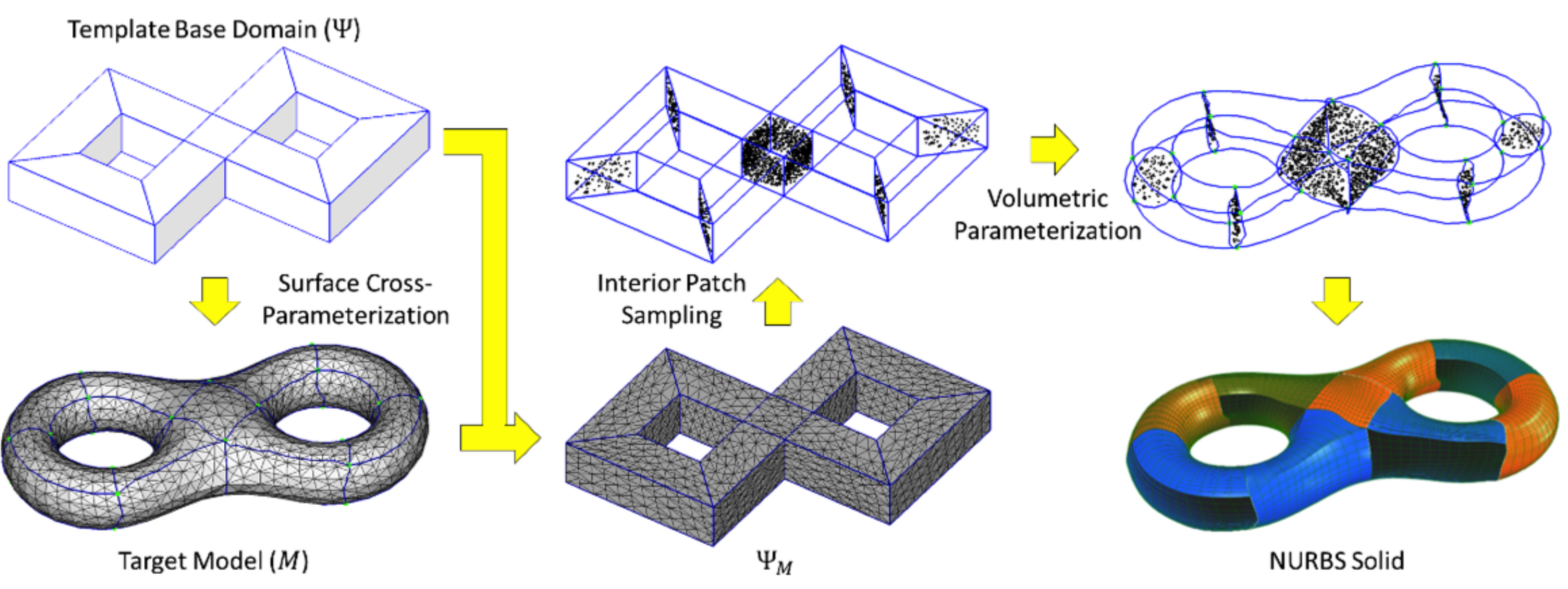}
\caption{Given the template base domain, $\Psi$, and a target surface
  model, $M$, the surface cross-parameterization \cite{KZW12} can be
  applied to partition $M$ consistently to the surface of $\Psi$. After obtaining
 a surface model $\Psi_M$ by mapping $M$ to  $\Psi$, the vertices on $\Psi_M$ and $M$ are used as handles to map the sampled points on the interior patches to $M$ by volumetric parameterization \cite{WHT07}. B-spline solid for $M$ can be obtained by fitting it to the boundary surface and the sampled points of $M$.} \label{fig:VolParaOverview}
\end{figure}

\section{Consistent B-spline volumetric parameterization of complex 3D
models}\label{sec:volcon}
To prepare for the computation reuse in IGA, we need to partition a model into a set of base
domains consistent to the pre-defined one, where each base domain will
be represented by a trivariate spline. We assume the reuse of IGA is applied to the models with similar
semantic features, e.g., a whole sequence of products having different
shapes but the same topology. The analysis will respect the semantic features, which are specified as anchor points, such that the boundary conditions can also be reused.

Given a connectivity of  volumetric base domains ($\Psi$), e.g., the
template base domain in Fig.\ref{fig:VolParaOverview}, we can
partition a target surface model ($M$) into a set of volumetric
sub-domains consistent to $\Psi$. The boundary surface is
first partitioned consistently according to anchor points \cite{KZW12,KZW:TVCG2012}, and the boundary surface is used
as the constraints to construct volumetric parameterization using
CSRBF-based elastic functions. The volumetric parameterization is used
to wrap the sample points on the interior surface of $\Psi$ to
$M$. The trivariate spline solids are constructed in each
domain  by fitting the boundary surface of $M$ and the sample points
on above determined interior surfaces.

\subsection{Consistent boundary surface decomposition}

Based on the anchor points specified on $M$, the surface of $M$ is partitioned consistently according to the nodes in $\Psi$ using the method in \cite{KZW12,KZW:TVCG2012}.
The basic idea is to trace shortest paths between anchor points on $M$
for each edge on the boundary surface of $\Psi$. To ensure topological
consistency, we need to ensure the paths are free of intersections,
blocking and wrong cyclical order when tracing the paths \cite{PSS01}.
Surface cross-parameterization \cite{KZW12,KZW:TVCG2012} can then be computed
through this partition, and a  surface model $\Psi_M$ is obtained by mapping the boundary surface of $M$ to $\Psi$.
Once the boundary surface is partitioned, the next step is to construct the interior patches.

\subsection{CSRBF-based volumetric parameterization}

Due to the reason that $M$ is a surface model without volumetric
information, we need to construct the interior patches for $M$.  A sampling is first taken on the interior patches in $\Psi$. By the
volumetric parameterization presented in \cite{WHT07}, using the
vertices on $\Psi_M$ and $M$ as handles, sample points on the
interior patches of $\Psi$ can be mapped to the interior of $M$.
Specifically, the volumetric parameterization can be expressed by the elastic function:
\begin{equation} \label{eq:CSRBF}
f(\mathbf{x}) = \sum_{j=1}^{n} d_j \phi(\mathbf{x} - \mathbf{v_j}) + P(\mathbf{x}),
\end{equation}
where $\mathbf{v_j}$s are the locations of the constraint vertices on $\Psi_M$ , $d_j$s are the weights, and $P(\mathbf{x})$ is a linear polynomial that accounts for the linear and constant portions of $f$. 

To solve for the set of $d_j$ that will satisfy the constraints,
$\mathbf{v_i'} = f(\mathbf{v_i})$, on the elastic function, we can substitute the right side of Eq.(\ref{eq:CSRBF}) for $f(\mathbf{v_i})$ and yield to
\begin{equation}
\mathbf{v_i'} = \sum_{j=1}^{k} d_j \phi(\mathbf{v_i} - \mathbf{v_j}) + P(\mathbf{x}).
\end{equation}
Since this equation is linear with respect to the unknowns
$d_j$s and the coefficients of $P(\mathbf{x})$, the unknowns can be formulated and solved as a linear system.

We take the Wendland's compactly supported radial basis function as $\phi(r)$ , it is given by
\begin{equation} \label{eq:CSRBF1}
\phi(r) = (1-r)_+^6(3 + 18r + 35r^2)\mbox{, with $r=\frac{\|\mathbf{\bar{x}}\|_2}{\lambda}$, $a_+ = max\{a,0\}$},
\end{equation}
which has a compact support with the radius $\lambda$, and has
$C^4$-continuity. As a compactly supported kernel function is used, 
the linear equation system will become sparse and can be efficiently solved by Cholesky decomposition or LU decomposition.

Now the volumetric parameterization has been established. The sample points on the interior patches in $\Psi$ are mapped to $M$ by $f(\cdot)$.

\subsection{B-spline solid construction}
 
The trivariate spline solid $S(u,v,w)$ (i.e., B-spline in our implementation) for each base domain
inside $M$ can be constructed by fitting the boundary surface of
$S(u,v,w)$ to its corresponding boundary surface on $M$ and the interior 
sample points obtained by the nonlinear elastic function $f(\cdot)$. By this way, we can convert the target model $M$ into a set of connected 
trivariate spline solids with consistent topology as $\Psi$. When
applying this to a sequence of models $\{M_i\}$, all models will have
spline solids in the same connectivity but different control points
(i.e., different shapes). With the help of this setup, we will show how to reuse the computation of IGA
taken on one model in the IGA of other models.

\section{Quadrature-free isogeometric analysis with \Bezier extraction and
 polynomial approximation}\label{sec:analysis}

In this section, an efficient quadrature-free method is proposed to
compute the entries of stiffness matrix with the help of \Bezier
extraction and polynomial approximation techniques of trivariate
rational \Bezier functions.  Here we use heat conduction problem as an
example to demonstrate the functionality of our approach. The
quadrature-free method can be applied to many other problems of
computational physics such as the linear elasticity problem in solid 
mechanics. 

\subsection{Preliminary on Bernstein polynomials}
Some preliminary on Bernstein polynomials will be used in our method. They are reviewed below. 

\begin{lemma} Product of Bernstein polynomials
\begin{equation}\label{eq:product1}
B_{i}^{m}(t) B_{j}^{n}(t) = \frac{{m \choose i}{n \choose j}}{{m+n \choose i+j}} B_{{i+j}}^{m+n}(t)
\end{equation}
\end{lemma}

\begin{lemma} Integration of Bernstein polynomials 
\begin{equation}\label{eq:integration}
\int_{0}^{1}B_{i}^{m}(t) dt= \frac{1}{m+1} 
\end{equation}
\end{lemma}

\begin{lemma} Degree elevation of Bernstein polynomials 
\begin{equation}
B_{i}^{n}(t) = t B_{i−1}^{n-1}(t) + (1-t) B_i^{n-1}(t)
\end{equation}
\end{lemma}

\begin{proposition}  \label{prop:prop}
Let $ R(u,v,w)$ and $S(u,v,w)$ be parametric function defined by 
\[
R(u,v,w)=\sum_{i=0}^{l_1}\sum_{j=0}^{m_1}\sum_{k=0}^{n_1} a_{ijk} B_i^{l_1}(u) B_j^{m_1}(v) B_k^{n_1}(w),
\]
and 
\[
S(u,v,w)=\sum_{i=0}^{l_2}\sum_{j=0}^{m_2}\sum_{k=0}^{n_2} b_{ijk} B_i^{l_2}(u) B_j^{m_2}(v) B_k^{n_2}(w),
\]
where $a_{ijk}$ and $b_{ijk}$ are scale values. Then the product of  $
R(u,v,w)$ and $S(u,v,w)$ can be defined as 
\begin{equation}\label{eq:product}
R(u,v,w) \times S(u,v,w)=\sum_{i=0}^{l_1+l_2}\sum_{j=0}^{m_1+m_2}\sum_{k=0}^{n_1+n_2}
c_{ijk} B_i^{l_1+l_2}(u) B_j^{m_1+m_2}(v) B_k^{n_1+n_2}(w),
\end{equation}
where 
\[
c_{ijk}=\sum_{r=\text{max}(0,i-l_1)}^{\text{min}(i,l_2)}\sum_{s=\text{max}(j-m_1)}^{\text{min}(j,m_2)}\sum_{t=\text{max}(0,k-n_1)}^{\text{min}(k,n_2)}\frac{{l_1 \choose
    r}{l_2 \choose i-r}{m_1 \choose s}{m_2 \choose j-s}{n_1 \choose t}{n_2
    \choose k-t}a_{rst}b_{{(i-r)},{(j-s)},{(k-t)}}}{{l_1+l_2 \choose i}{m_1+m_2 \choose j}{n_1+n_2 \choose k}}
\]
\end{proposition}

This proposition can be proved directly by Eq.(\ref{eq:product1}). 

\subsection{ \Bezier extraction of B-spline volume }

In order to achieve an efficient computation, the isogeometric analysis
problem is solved with \Bezier extraction \cite{Borden}, in which  
piece-wise B-spline representation is first converted into a \Bezier form. 

Without loss of generality, B-spline basis defined on a knot vector
can be written as a linear combination of the
Bernstein polynomials, that is
\begin{equation}\label{eqn:extraction1}
\mathbf {N}(\mathbf{\xi}) = \mathbf{C}(\mathbf{\xi}) \mathbf{B}(\mathbf{\xi})
\end{equation}
where $\mathbf{C}(\mathbf{\xi}) $ denotes the \Bezier extraction operator and $ B(\bf
\xi)$ are the Bernstein polynomials which are defined on $[0, 1]$. 
The conversion matrix $\mathbf{C}(\mathbf{\xi}) $ is sparse and its entries can be
obtained by knot insertions and recursive computation. 
Details on \Bezier extraction can be found in Borden et al. \cite{Borden} and
Scott et al. \cite{Scott}.

With the conversion matrix $\mathbf{C}(\mathbf{\xi})
$, $\mathbf{C}(\mathbf{\eta}) $ and $\mathbf{C}(\mathbf{\zeta}) $ , the \Bezier extraction of B-spline
volume can be represented as follows
\begin{equation}
\mathbf{P} = (\mathbf{C}(\xi) \otimes\mathbf{C}(\eta)\otimes \mathbf{C}(\zeta)) \mathbf{Q}
\end{equation}
where $\mathbf{Q}$ are the control points of the B-spline volume, 
$\mathbf{P}$ are the control points of the extracted \Bezier volume,
$\mathbf{C}(\xi)$, $\mathbf{C}(\eta)$ and $\mathbf{C}(\zeta)$ are
derived from (\ref{eqn:extraction1}) . 
An example with cubic B-spline volume is shown in Fig.\ref{fig:extraction}  to illustrate the 
extraction results and the corresponding control lattice of the
extracted four \Bezier volumes with different colors.  

\begin{figure*}
\centering
\begin{minipage}[t]{3.1in}
\centering
\includegraphics[width=2.7in]{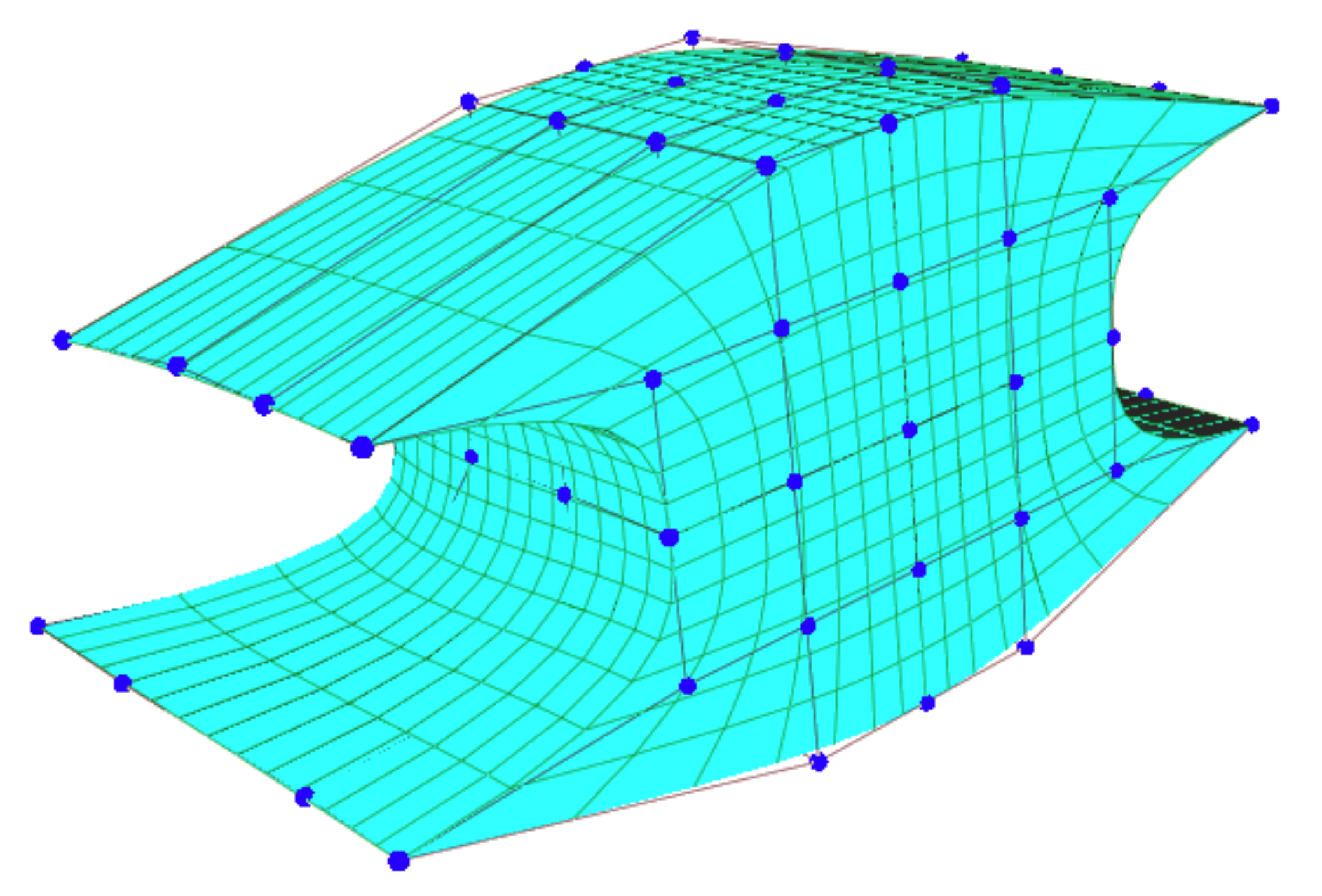}
\\ (a) Cubic B-spline volume
\end{minipage}
\begin{minipage}[t]{3.1in}
\centering
\includegraphics[width=2.7in ]{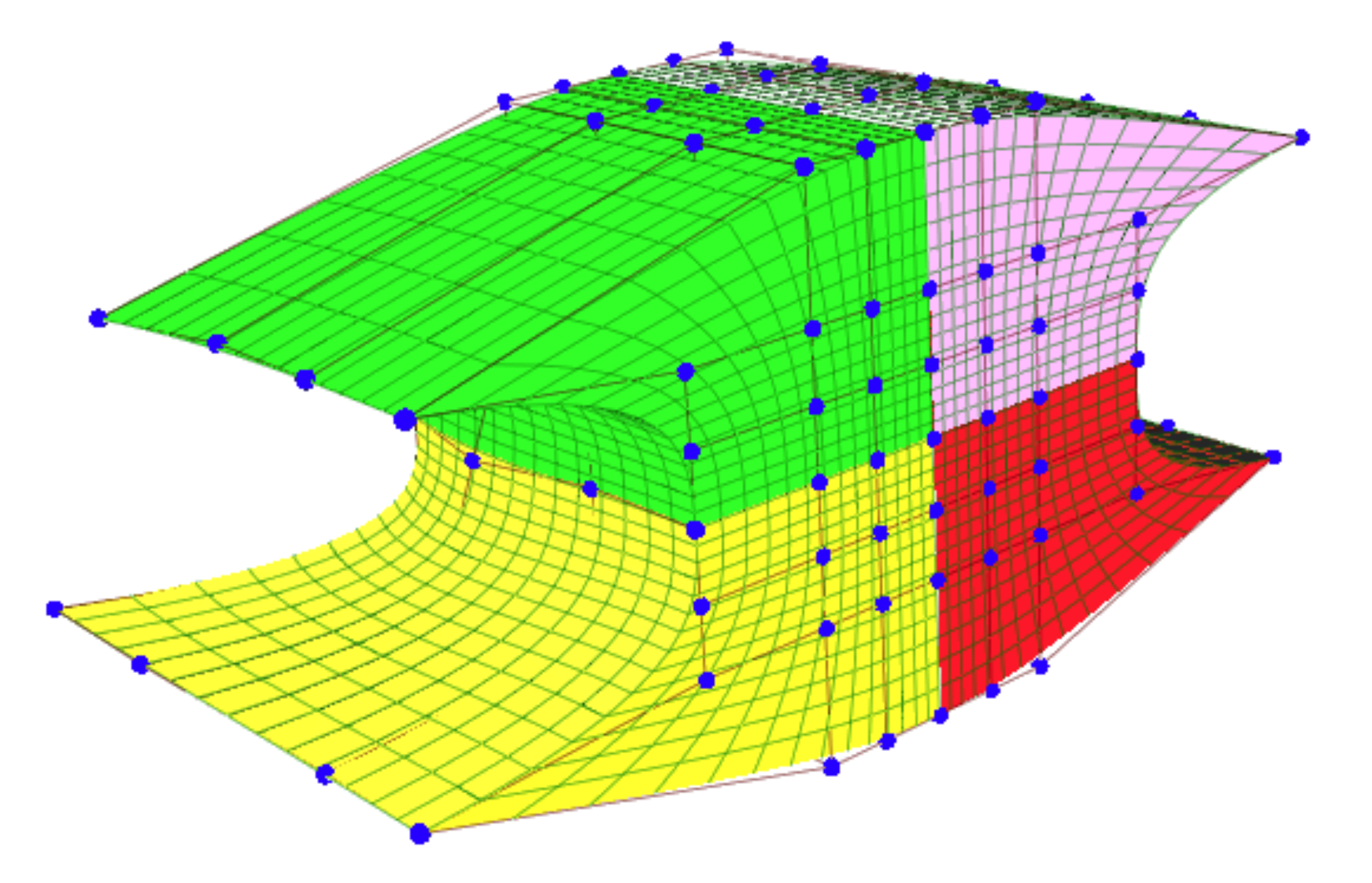}
\\ (b) Extracted \Bezier volume of B-spline volume in (a).
\end{minipage} 
\caption{Example of \Bezier extraction.}
\label{fig:extraction}
\end{figure*}

\subsection{Isogeometric analysis of heat conduction problem }

Given a domain $\Omega$ with boundary $\Gamma =\partial \Omega_D  $
and a volumetric parameterization as follows
\[
\Sc(\xi,\eta, \zeta)=(x(\xi,\eta, \zeta),
y(\xi,\eta,\zeta),  z(\xi,\eta,\zeta))=\sum_{i=0}^n \sum_{j=0}^m
\sum_{k=0}^l B_i^p(\xi) B_j^q(\eta) B_k^r(\zeta)
\pb_{i,j,k}. 
\]
We consider the following thermal conduction problem:
\begin{equation} 
\begin{split}
& {\boldsymbol \nabla} ( {\boldsymbol \nabla} T ({\bf x}) )= g({\bf x})  \quad \text{ in } \Omega \\
& T({\bf x}) = T_0 ({\bf x}) \text{ on } \partial \Omega_D 
\end{split}
\label{eq:heat}
\end{equation}
where ${\bf x}$ are the Cartesian coordinates, $T$ represents
the temperature field,  Dirichlet condition with imposed temperature $T_0$ is
applied on $\partial \Omega_D$, and $g$ is a user-defined function as a
source term to the classical heat conduction equation.

According to a classical variational approach, we seek for a
solution $T$ in the Sobolev space $ H^1(\Omega)$, such as $ T({\bf x}) = T_0 ({\bf x})$ on $\partial \Omega_D$ and
\begin{equation*}
\int_\Omega {\boldsymbol \nabla} (
    {\boldsymbol \nabla} T ({\bf x}) ) \; \psi({\bf x}) \; d \Omega = \int_\Omega g({\bf x})  \; \psi({\bf x}) \; d \Omega
    \qquad \forall \psi \in H^1_{\partial \Omega_D}(\Omega),
\end{equation*}
where $\psi({\bf x})$ are trial functions. After integrating by parts
and applying the boundary conditions, we can obtain
\begin{equation}
- \int_\Omega  
    {\boldsymbol \nabla} T ({\bf x}) \; {\boldsymbol \nabla}
    \psi({\bf x}) \; d \Omega  =  \int_\Omega f({\bf x})  \; \psi({\bf x}) \; d \Omega.
    \label{eq:weak_formulation}
\end{equation}
Following the paradigm of IGA, the temperature field is represented by
using trivariate spline  basis functions as 
\begin{equation*}
T(\xi,\eta, \zeta) = \sum_{i=1}^{n_i} \, \sum_{j=1}^{n_j} \,
\sum_{k=1}^{n_k} \hat{B}_i^{p_i} (\xi) \, \hat{B}_j^{p_j} (\eta) \,
\hat{B}_k^{p_k} (\zeta) \, T_{ijk} ,
\label{eq:T-iso}
\end{equation*}
where $\hat{B}_i$ functions are Bernstein basis functions and
$\ub=(\xi,\eta, \zeta)\in\Pc$ are domain parameters. Then, we define
the trial  functions $\psi({\bf x})$ in the physical domain as
\begin{equation*}
B_{ijk} ({\bf x}) = \hat{B}_{ijk} \circ \sigma^{-1} (x,y,z) = \hat{B}_{ijk}
(\xi,\eta, \zeta) = \hat{B}_i^{p_i} (\xi) \, \hat{B}_j^{p_j} (\eta) \,\hat{B}_k^{p_k} (\zeta) .
\end{equation*}
The weak form in Eq.~(\ref{eq:weak_formulation}) can then be written as
\begin{equation*}
\label{weak_iso}
\begin{split}
\sum_{r=1}^{n_r} \sum_{s=1}^{n_s} \sum_{t=1}^{t_l} T_{rst} \int_{\Omega}  
    {\boldsymbol \nabla} B_{rst} ({\bf x})  \; {\boldsymbol \nabla}
   B_{ijk} ({\bf x})  \; d \Omega  =  -
       \int_\Omega g({\bf x})  \; B_{ijk} ({\bf x}) \; d \Omega.
\end{split}
\end{equation*}
Finally, we obtain a linear system with the coefficient matrix similar
to the stiffness matrix obtained from the
classical finite-element methods, such as 
\begin{equation}
\sum_{r=1}^{n_r} \sum_{s=1}^{n_s} \sum_{t=1}^{t_l} T_{rst} M_{ijk,rst}=S_{ijk}
\end{equation}
with 
\begin{equation}
\begin{split}
 M_{ijk,rst} & =  \int_\Omega 
    {\boldsymbol \nabla} B_{rst} ({\bf x})  \; {\boldsymbol \nabla}
   B_{ijk} ({\bf x})  \; d \Omega \\
   &=  \int_{\Pc}  
    {\boldsymbol \nabla}_{\ub} \hat{B}_{rst} ^T ({\ub}) J(\ub)
    {\boldsymbol K} ^T (\ub) {\boldsymbol K}(\ub) \;{\boldsymbol \nabla}_{\ub} \hat{B}_{ijk} ({\ub}) \; d \Pc\\
 S_{ijk} &=   \int_\Omega g({\bf x})  \; B_{ijk} ({\bf x}) \;
 d \Omega \\
   &=     \int_\Pc  \hat{B}_{ijk} ({\ub}) J(\ub) g({T(\ub)}) \; d \Pc.
\label{eq:system_iso}
\end{split}
\end{equation}
Here $J$ is the Jacobian of the transformation from physical domain
to parametric domain, and $\boldsymbol K$ is the inverse of the
Jacobian matrix. 

\subsection{Explicit representation of stiffness matrix entries}
\label{subsection:stiffness}
Suppose that the entries of the stiffness matrix are denoted by $ M_{ijk,rst}  =
\int_{0}^{1}\int_{0}^{1}\int_{0}^{1}  F(\ub) d\ub $, we can derive the
following proposition. 
\begin{proposition} \label{propostion:stiffness}
$F(\ub)$ can be represented as a trivariate rational \Bezier function
as follows
\begin{equation}\label{eqn:F(u)}
F(\ub)=\frac{\sum_{i=0}^{6l-4}\sum_{j=0}^{6m-4}\sum_{k=0}^{6n-4} F_{ijk} B_i^{6l-4}(u) B_j^{6m-4}(v) B_k^{6n-4}(w)}{\sum_{i=0}^{3l-1}\sum_{j=0}^{3m-1}\sum_{k=0}^{3n-1} J_{ijk} B_i^{3l-1}(u) B_j^{3m-1}(v) B_k^{3n-1}(w)}.
\end{equation}
\end{proposition}
\begin{proof}
\begin{eqnarray}
F(\ub) & = & {\boldsymbol \nabla}_{\ub}\hat{B}_{rst} ^T
 ({\ub}) J(\ub) {\boldsymbol K} ^T (\ub) {\boldsymbol K}(\ub) \;{\boldsymbol \nabla}_{\ub}
 \hat{B}_{ijk} ({\ub}) \nonumber \\
   & = &  J(\ub)\left({\begin{array}{ccc}
      \frac{\partial \hat{B}_{rst}({\ub})}{\partial \xi}  &
       \frac{\partial \hat{B}_{rst}({\ub})}{\partial \eta}   &
        \frac{\partial \hat{B}_{rst}({\ub})}{\partial \zeta}  
 \end{array}} \right)  \left({\begin{array}{ccc}
       \frac{\partial \xi}{\partial x} & \frac{\partial \eta}{\partial x} & \frac{\partial \zeta}{\partial x}  \\
       \frac{\partial \xi}{\partial y} & \frac{\partial \eta}{\partial y} & \frac{\partial \zeta}{\partial y}  \\
       \frac{\partial \xi}{\partial z} & \frac{\partial \eta}{\partial
         z} & \frac{\partial \zeta}{\partial z}  \\\end{array}}
 \right)\left({\begin{array}{ccc}
       \frac{\partial \xi}{\partial x} & \frac{\partial \xi}{\partial y} & \frac{\partial \xi}{\partial z}  \\
       \frac{\partial \eta}{\partial x} & \frac{\partial \eta}{\partial y} & \frac{\partial \eta}{\partial z}  \\
       \frac{\partial \zeta}{\partial x} & \frac{\partial \zeta}{\partial y} & \frac{\partial \zeta}{\partial z}  \\\end{array}} \right)\left({\begin{array}{c}
      \frac{\partial \hat{B}_{ijk}({\ub})}{\partial \xi}  \\
       \frac{\partial \hat{B}_{ijk}({\ub})}{\partial \eta}   \\
        \frac{\partial \hat{B}_{ijk}({\ub})}{\partial \zeta}  \\
 \end{array}} \right)  \nonumber \\
&=& J(\ub)\left({\begin{array}{ccc}
      \frac{\partial \hat{B}_{rst}({\ub})}{\partial \xi}  &
       \frac{\partial \hat{B}_{rst}({\ub})}{\partial \eta}   &
        \frac{\partial \hat{B}_{rst}({\ub})}{\partial \zeta}  
 \end{array}} \right)\left({\begin{array}{ccc}
      a & b  & c  \\
      b & d  & e  \\
      c & e  & f  \\
 \end{array}} \right) \left({\begin{array}{c}
      \frac{\partial \hat{B}_{ijk}({\ub})}{\partial \xi}  \\
       \frac{\partial \hat{B}_{ijk}({\ub})}{\partial \eta}   \\
        \frac{\partial \hat{B}_{ijk}({\ub})}{\partial \zeta}  \\
 \end{array}} \right)\nonumber \\
&=& J(\ub) ( a \frac{\partial \hat{B}_{rst}({\ub})}{\partial
  \xi}\frac{\partial \hat{B}_{ijk}({\ub})}{\partial
  \xi}+d\frac{\partial \hat{B}_{rst}({\ub})}{\partial \eta}\frac{\partial \hat{B}_{ijk}({\ub})}{\partial \eta}
+f\frac{\partial \hat{B}_{rst}({\ub})}{\partial \zeta}\frac{\partial
  \hat{B}_{ijk}({\ub})}{\partial \zeta}\\ &&+2b \frac{\partial \hat{B}_{rst}({\ub})}{\partial
  \xi}\frac{\partial \hat{B}_{ijk}({\ub})}{\partial
  \eta}+2c \frac{\partial \hat{B}_{rst}({\ub})}{\partial
  \xi}\frac{\partial \hat{B}_{ijk}({\ub})}{\partial
  \zeta}+ 2e \frac{\partial \hat{B}_{rst}({\ub})}{\partial
  \eta}\frac{\partial \hat{B}_{ijk}({\ub})}{\partial
  \zeta})\nonumber 
\end{eqnarray}
in which 
\begin{eqnarray*} 
      a = \xi^{2}_{x}+ \eta^{2}_{x}+  \zeta^{2}_{x} & b = \xi_x\xi_y
      +\eta_x\eta_y+\zeta_x\zeta_y  &
 c = \xi_x\xi_z +\eta_x\eta_z+\zeta_x\zeta_z  \\ d=\xi^{2}_{y}+
 \eta^{2}_{y}+  \zeta^{2}_{y} &
e= \xi_y\xi_z +\eta_y\eta_z+\zeta_y\zeta_z   & f =  \xi^{2}_{z}+ \eta^{2}_{z}+  \zeta^{2}_{z} 
\end{eqnarray*}
\[
\xi_x=\frac{y_\eta z_\zeta-y_\zeta z_\eta}{J}, \xi_y=-\frac{x_\eta z_\zeta-x_\zeta z_\eta}{J},  \xi_z=\frac{y_\eta x_\zeta-y_\zeta x_\eta}{J}, 
\]
\[
\eta_x=\frac{y_\xi z_\zeta-y_\zeta z_\xi}{J}, \eta_y=-\frac{x_\xi z_\zeta-x_\zeta z_\xi}{J},  \eta_z=\frac{y_\xi x_\zeta-y_\zeta x_\xi}{J}, 
\]
\[
\zeta_x=\frac{y_\eta z_\xi-y_\xi z_\eta}{J}, \zeta_y=-\frac{x_\eta z_\xi-x_\xi z_\eta}{J},  \zeta_z=\frac{y_\eta x_\xi-y_\xi
  x_\eta}{J}. 
\]
\begin{eqnarray}
J ({\ub}) &= &\left|{\begin{array}{ccc}
       x_\xi & y_\xi & z_\xi  \\
       x_\eta & y_\eta & z_\eta  \\
       x_\zeta & y_\zeta & z_\zeta  \\
 \end{array}} \right| =
\sum_{i=0}^{3l-1}\sum_{j=0}^{3m-1}\sum_{k=0}^{3n-1} J_{ijk}
B_i^{3l-1}(u) B_j^{3m-1}(v) B_k^{3n-1}(w), 
\label{eq:J}
\end{eqnarray}
in which  $J_{ijk}$ has the following form as given in \cite{qian:cad}
\begin{equation}\label{jacobian}
J_{ijk}=\sum\limits_{\substack{i_1+i_2+i_3=i\\i_1\in[0,l-1]\\i_2\in[0,l]\\i_3\in[0,l]}}\sum\limits_{\substack{j_1+j_2+j_3=j\\j_1\in[0,m]\\j_2\in[0,m-1]\\j_3\in[0,m]}}\sum\limits_{\substack{k_1+k_2+k_3=k\\k_1\in[0,n]\\k_2\in[0,n]\\k_3\in[0,n-1]}}
D_{ijk} \cdot
 \mathbf{det} \left({\begin{array}{c}
      \pb_{i_1+1,j_1,k_1}-\pb_{i_1,j_1,k_1}\\
      \pb_{i_2,j_2+1,k_2}-\pb_{i_2,j_2,k_2}  \\
      \pb_{i_3,j_3,k_3+1}-\pb_{i_3,j_3,k_3}) \\
 \end{array}} \right)^T,     
\end{equation}
with
\begin{equation}\label{Dijk}
D_{ijk}=lmn\frac{{l-1 \choose i_1}{l \choose i_2}{l \choose i_3}{m \choose
    j_1}{m-1 \choose j_2}{m \choose j_3}{n \choose k_1}{n \choose
    k_2}{n-1 \choose k_3}}{{3l-1 \choose i}{3m-1 \choose j}{3n-1
    \choose k}}. 
\end{equation}

\noindent According to Eq.(8) the product formula of two trivariate
Bernstein polynomials in Proposition \ref{prop:prop}, we can rewritten $F(\ub)$ as a high-order 
trivariate rational Bernstein polynomial, 
\begin{equation}\label{eqn:entry}
F(\ub) =\frac{\sum_{i=0}^{6l-4}\sum_{j=0}^{6m-4}\sum_{k=0}^{6n-4} F_{ijk} B_i^{6l-4}(u) B_j^{6m-4}(v) B_k^{6n-4}(w)}{\sum_{i=0}^{3l-1}\sum_{j=0}^{3m-1}\sum_{k=0}^{3n-1} J_{ijk} B_i^{3l-1}(u) B_j^{3m-1}(v) B_k^{3n-1}(w)}
\end{equation}
in which $F_{ijk}$ can be computed according to
Eq.(\ref{eq:product}). 
\end{proof}

In general cases, the integration of a rational \Bezier function over
$[0,1]$ is either very  complex or has no analytically solution. 
Gaussian-quadrature method is usually employed in general IGA to 
compute the integration of rational function in Eq.(\ref{eq:system_iso})
approximately. As shown in Eq.(\ref{eq:integration}), the integration of polynomial
\Bezier functions (non-rational) has an explicit
and exact form. As a classical problem in CAGD,  approximating rational \Bezier curves and surfaces
with polynomial \Bezier curves and surfaces has been studied in
\cite{Hu:cam2013,Shimao:2015,Wang:jat1997}. In this paper,  we
further extend the weighted least-squares approach \cite{Shimao:2015} to trivariate splines and
approximate the trivariate rational function $F(\ub)$ with
a trivariate polynomial \Bezier function
$G(\ub)=\sum_{i=0}^{\alpha}\sum_{j=0}^{\beta}\sum_{k=0}^{\gamma}
G_{ijk} B_i^{\alpha}(u) B_j^{\beta}(v) B_k^{\gamma}(w)$.

Suppose that $F(\ub)$ can be rewritten as $F(\ub)=\frac{\D F_1(u,v,w)}{\D F_2(u,v,w)}$, 
our \Bezier approximation problem can be stated as that  to find the
control variables $G_{ijk}$, which can make the trivariate \Bezier
representation $G(u,v,w)=\sum_{i=0}^{\alpha}\sum_{j=0}^{\beta}\sum_{k=0}^{\gamma}
G_{ijk} B_i^{\alpha}(u) B_j^{\beta}(v) B_k^{\gamma}(w)$ become the best
approximation of $F(\ub)$. That is, to minimize the following
objective function when $F_2(u,v,w)$ is set to be the weight function
\[
D(F,G)=\int_0^1\int_0^1\int_0^1 F_2(u,v,w)(\frac{F_1(u,v,w)}{F_2(u,v,w)}-G(u,v,w))^2 dudvdw
\] 
Hence, the coefficient  $G_{ijk}$  can be obtained by letting 
\begin{equation}\label{eq:gijk}
\frac{\partial D(F,G)}{\partial G_{ijk}} =0, 
\end{equation} 
which is 
\begin{equation}\label{eq:gijk}
2 \int_0^1\int_0^1\int_0^1(F_1(u,v,w)-F_2(u,v,w)G(u,v,w)) B_i^{\alpha}(u) B_j^{\beta}(v)
B_k^{\gamma}(w) dudvdw =0. 
\end{equation} 

\noindent Eq.(\ref{eq:gijk}) can be rewritten as 
\begin{eqnarray*}
\int_0^1\int_0^1\int_0^1 F_2(u,v,w)G(u,v,w) B_i^{\alpha}(u) B_j^{\beta}(v)
B_k^{\gamma}(w) dudvdw  \\
= \int_0^1\int_0^1\int_0^1 F_1(u,v,w) B_i^{\alpha}(u) B_j^{\beta}(v)
B_k^{\gamma}(w) dudvdw 
\end{eqnarray*}

From the product and integral computation properties of \Bezier polynomials, we have 
\begin{eqnarray}\label{eq:linear}
 \sum_{a=0}^{\alpha+3l-1}\sum_{b =
  0}^{\beta+3m-1}\sum_{c=0}^{\gamma+3n-1} \frac{\T H_{a,b,c}}{\T {2\alpha+3l-1 \choose a+i}{2\beta+3m-1
    \choose b+j}{2\gamma+3n-1 \choose c+k}}  \\
= \sigma \sum_{p=0}^{6l-4}\sum_{q = 0}^{6m-4}\sum_{r=0}^{6n-4} \frac{\T {6l-4 \choose
    p} {6m-4 \choose
    q} {6n-4 \choose
    r} }{\T {\alpha+6l-4 \choose p+i}{\beta+6m-4 \choose
    q+j}{\gamma+6n-4 \choose r+k}}  F_{p,q,r}, 
\end{eqnarray}
in which 
\begin{equation}\label{sigma}
\sigma=
\frac{(2\alpha+3l)(2\beta+3m)(2\gamma+3n)}{(\alpha+6l-3)(\beta+6m-3)(\gamma+6n-3)}, 
\end{equation}
\[
H_{a,b,c}=\sum_{r=max(0,a-\alpha)}^{min(a,3l-1)}\sum_{s=max(0,b-\beta)}^{min(b,3m-1)}\sum_{t=max(0,c-\gamma)}^{min(c,3n-1)}\T {3l-1 \choose
    r}{\alpha \choose a-r}{3m-1 \choose s}{\beta \choose b-s}{3n-1 \choose t}{\gamma
    \choose c-t}J_{rst}G_{{(a-r)},{(b-s)},{(c-t)}}
\]

\noindent Then, from Eq.(\ref{eq:linear}) for all the $G_{ijk}$, we
can obtain a linear system  in the following form,   
\begin{equation}\label{system}
{\bf L}\cdot {\bf E} \cdot {\bf G}=\sigma \cdot {\bf Q} \cdot {\bf F}, 
\end{equation}
in which $\mathbf{G} =[G_{ijk}]_{(\alpha+1) \times (\beta+1) \times (\gamma+1)}$
is a vector with unknown
variables as entries, $\sigma$ is defined in Eq.(\ref{sigma}), 
$\mathbf{L}=
[L_{ijk}^{abc}]$ is a matrix with  
\begin{equation}\label{eqn:L}
L_{ijk}^{abc}=\frac{\T 1}{\T {2\alpha+3l-1 \choose a+i}{2\beta+3m-1
    \choose b+j}{2\gamma+3n-1 \choose c+k}},  
\end{equation}
$\mathbf{E}=
[E_{rst}^{abc}]$ is a matrix with
\begin{equation}\label{eqn:E}
E_{rst}^{abc}= {3l-1 \choose
    r}{\alpha \choose a-r}{3m-1 \choose s}{\beta \choose b-s}{3n-1 \choose t}{\gamma
    \choose c-t}J_{rst}, 
\end{equation}
$\mathbf{Q}=
[Q_{ijk}^{pqr}]$ is a matrix with
\begin{equation}\label{eqn:Q}
Q_{ijk}^{pqr}=  \frac{\T {6l-4 \choose
    p} {6m-4 \choose
    q} {6n-4 \choose
    r} }{\T {\alpha+6l-4 \choose p+i}{\beta+6m-4 \choose
    q+j}{\gamma+6n-4 \choose r+k}}, 
\end{equation}
and $\mathbf{F}=[F_{pqr}]$ is a vector with $F_{pqr}$ as entries defined
in Eq.(\ref{eqn:entry}). 

By solving this linear system, the \Bezier approximation  $G(u,v,w)$ of
the rational function $F(\ub)$ can be obtained as 
\begin{equation}\label{solution}
{\bf G}=\sigma \cdot {\bf E}^{-1}\cdot {\bf L}^{-1} {\bf Q}
\cdot {\bf F}. 
\end{equation}
By using Eq.(\ref{eq:integration}) in Lemma 4.2, the entries
$M_{ijk,rst}$  of the stiffness matrix can be evaluated by the following explicit form
\begin{equation}
M_{ijk,rst}=\int_{0}^{1}\int_{0}^{1}\int_{0}^{1}  F(\ub) d\ub \approx
\int_{0}^{1}\int_{0}^{1}\int_{0}^{1}  G(\ub) d\ub 
=\frac{1}{(\alpha+1)(\beta+1)(\gamma+1)}\sum_{i=0}^{\alpha}\sum_{j=0}^{\beta}\sum_{k=0}^{\gamma}G_{ijk}. 
\end{equation}

\begin{figure}[!t]
\centering
\begin{minipage}[t]{1.0\figwidth}
\centering
\includegraphics[ width=1.0\textwidth]{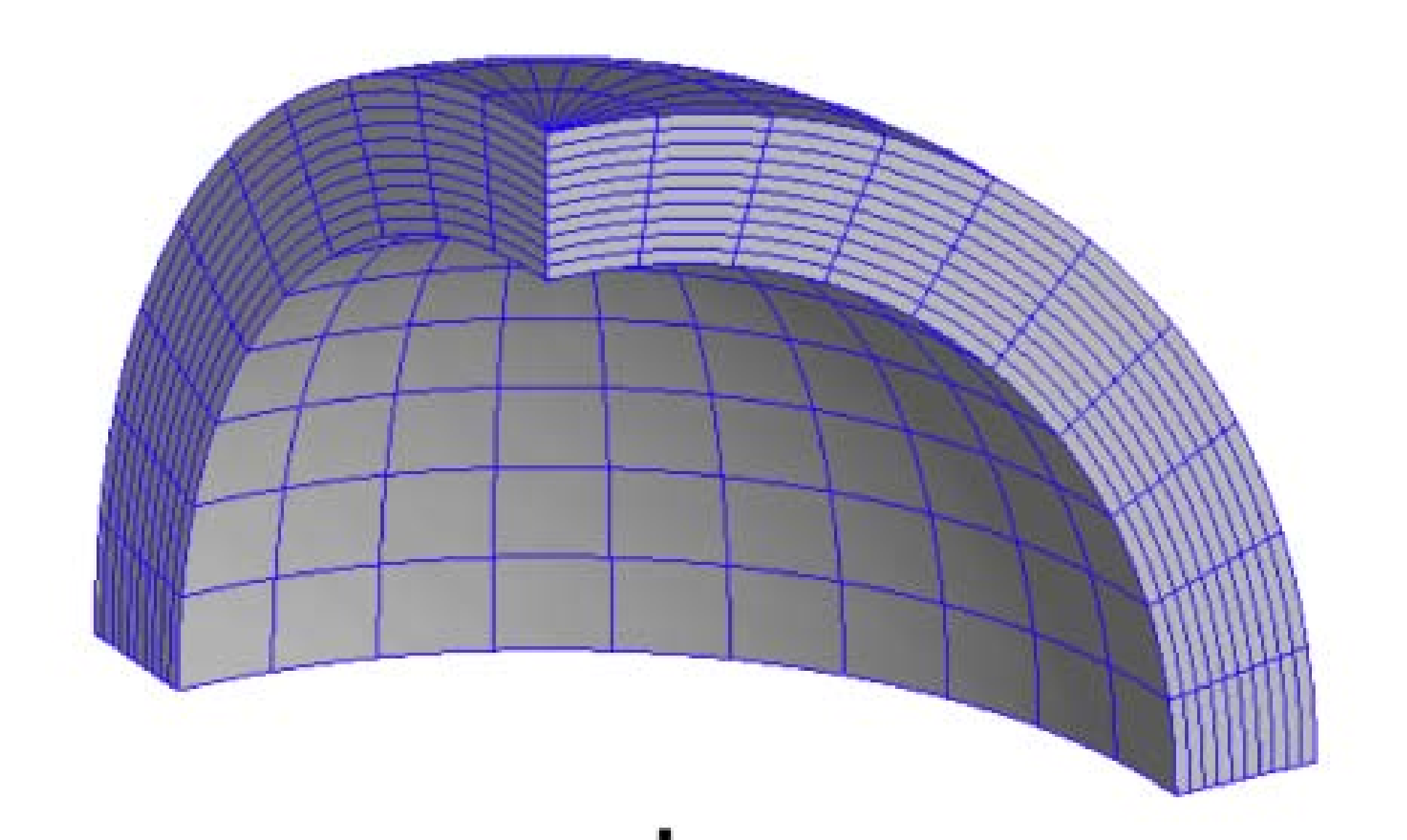}
\\ (a)  1/8th of a hollow sphere
\end{minipage}
\begin{minipage}[t]{1.0\figwidth}
\centering
\includegraphics[ width=1.0\textwidth]{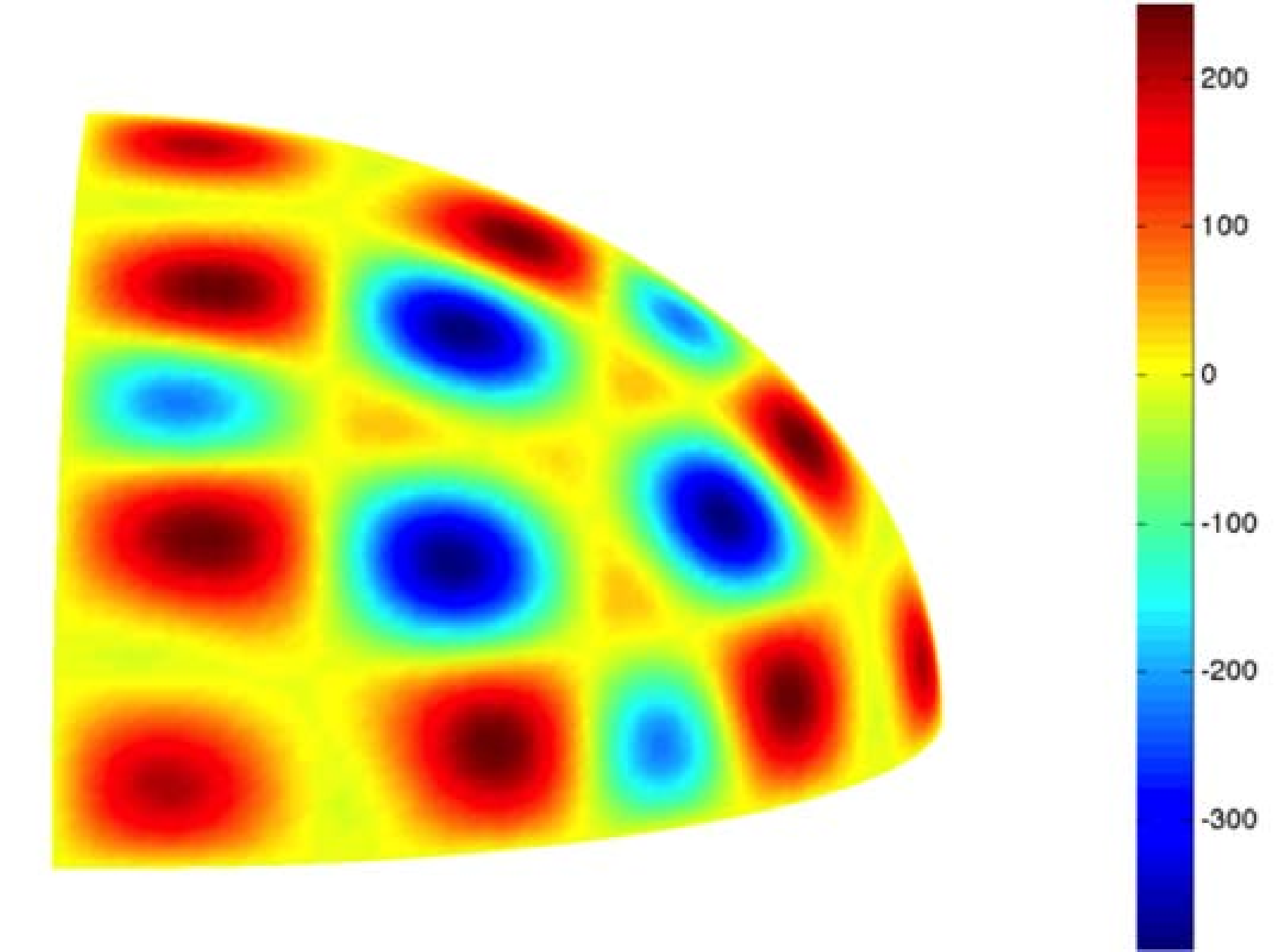}
\\ (b)  exact solution
\end{minipage}\\
\begin{minipage}[t]{1.0\figwidth}
\centering
\includegraphics[  width=1.\textwidth]{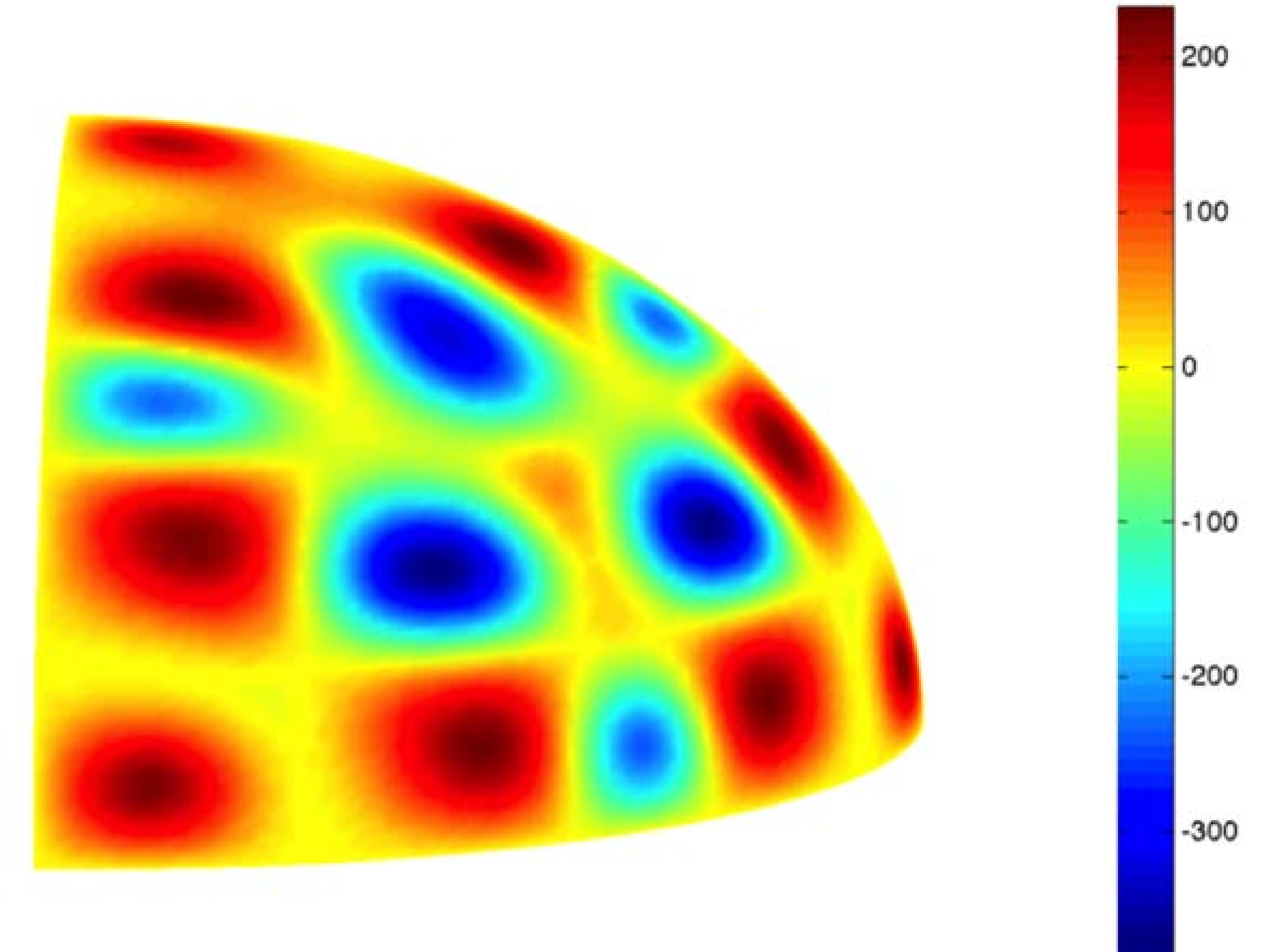}
\\ (c)  IGA-Galerkin solution
\end{minipage} 
\begin{minipage}[t]{1.0\figwidth}
\centering
\includegraphics[ width=1. \textwidth]{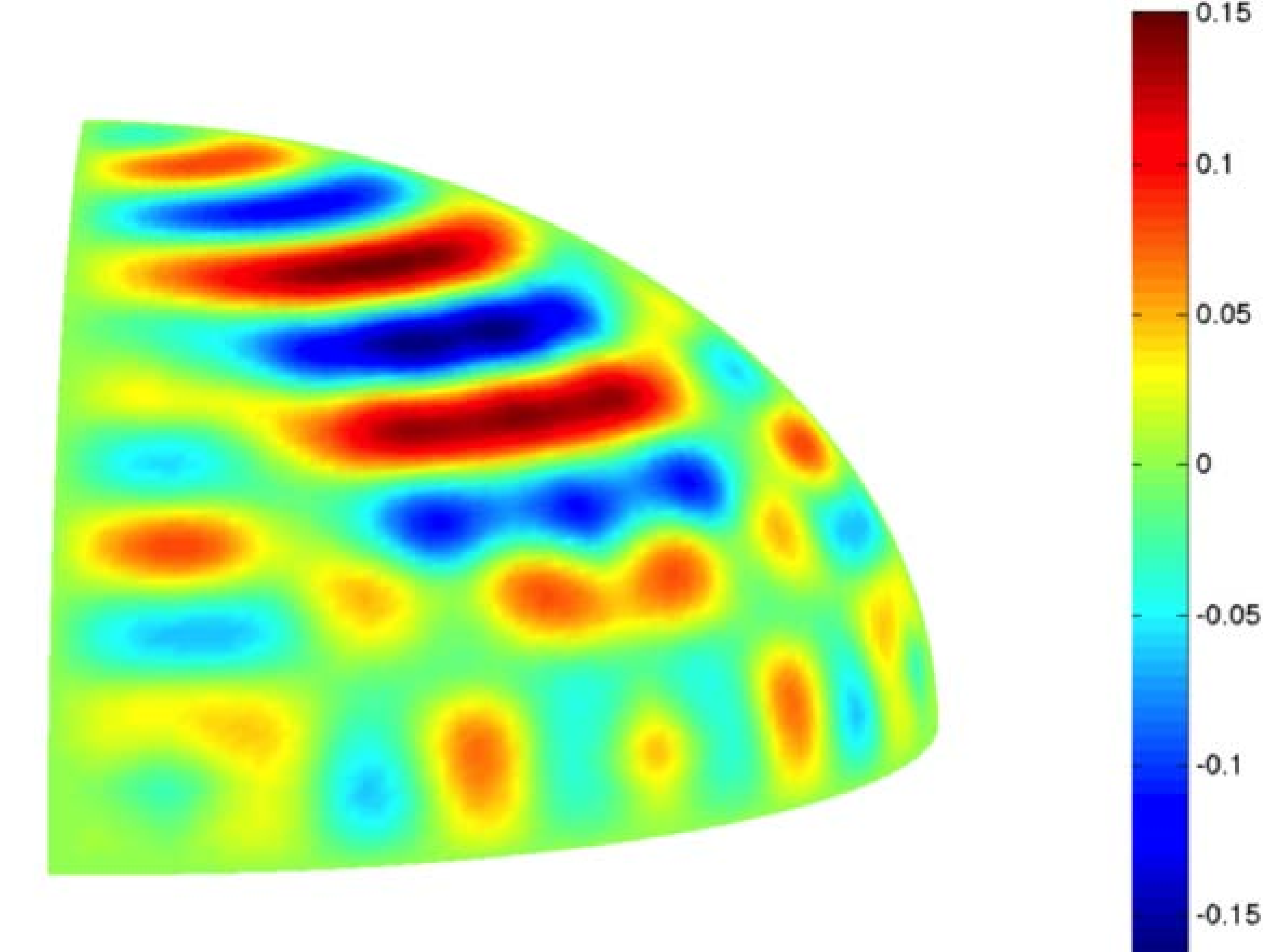}
\\ (d)  IGA-Galerkin errors 
\end{minipage}\\
\begin{minipage}[t]{1.0\figwidth}
\centering
\includegraphics[ width=1.\textwidth]{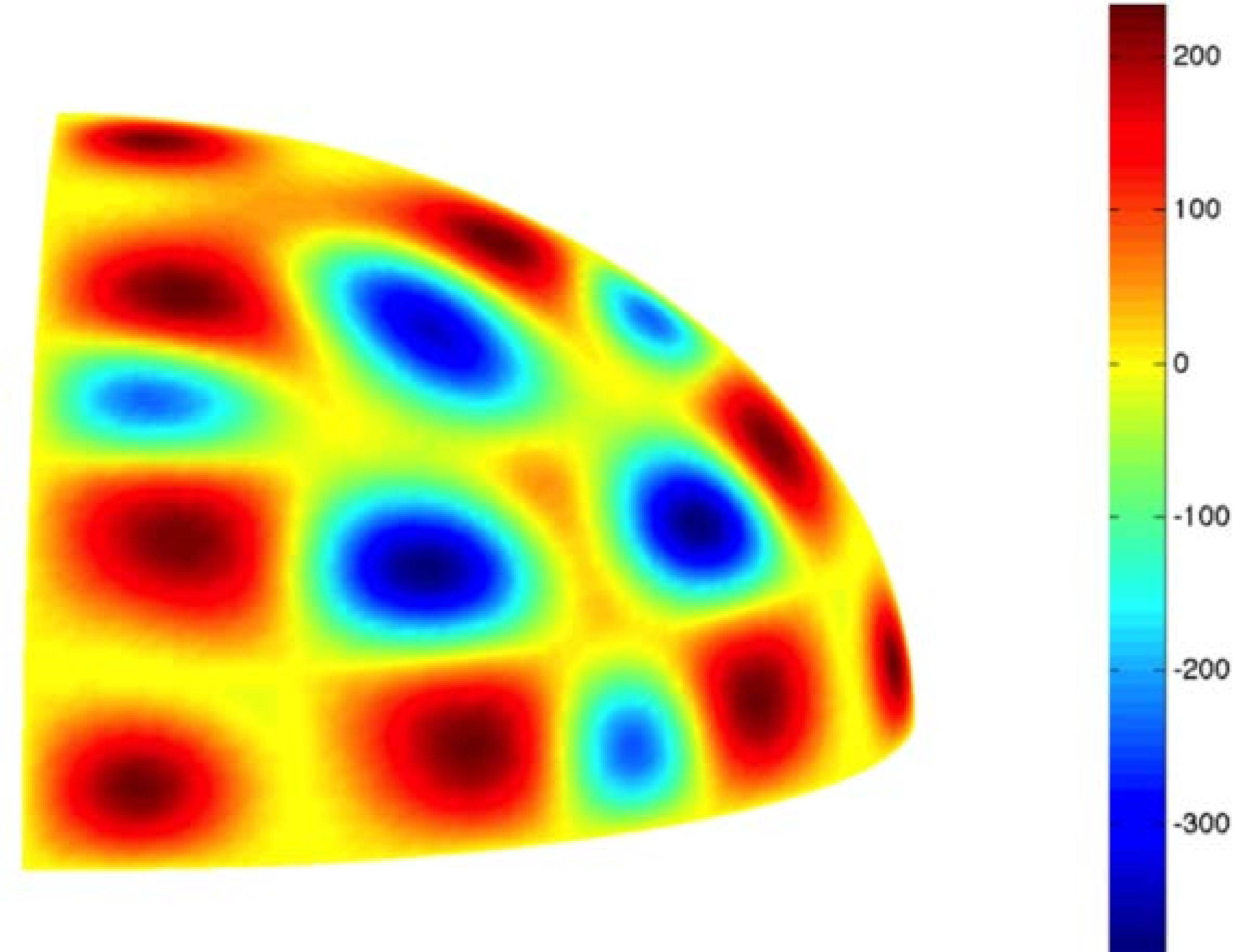}
\\ (e)  solution w.r.t the proposed method   
\end{minipage} 
\begin{minipage}[t]{1.0\figwidth}
\centering
\includegraphics[ width=1.\textwidth]{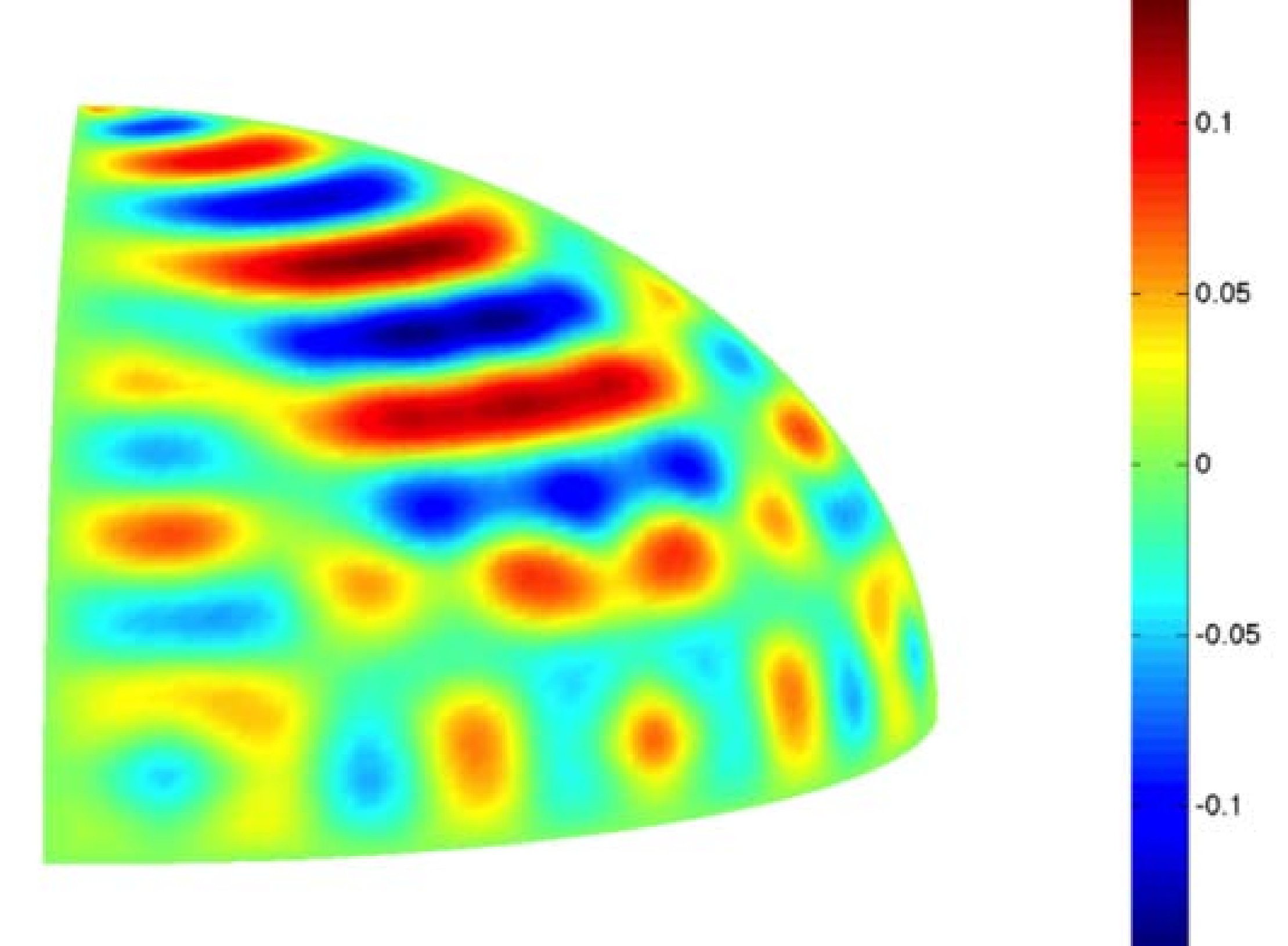}
\\ (f)  error w.r.t the proposed method   
\end{minipage}\\
\caption{Hollow sphere model problem with the solution and errors
  plotted on the isoparametric surface with $w=0.6$. (a) volume
  parameterization; (b) exact solution; (c) IGA-Galerkin solution; 
(d) IGA-Galerkin error; (e) solution of the proposed quadrature-free
method; (f) error of the proposed quadrature-free method. }
\label{fig:simulation}
\end{figure}

\begin{figure}[!t]
\centering
\begin{minipage}[t]{1.5\figwidth}
\centering
\includegraphics[width=1.5\figwidth]{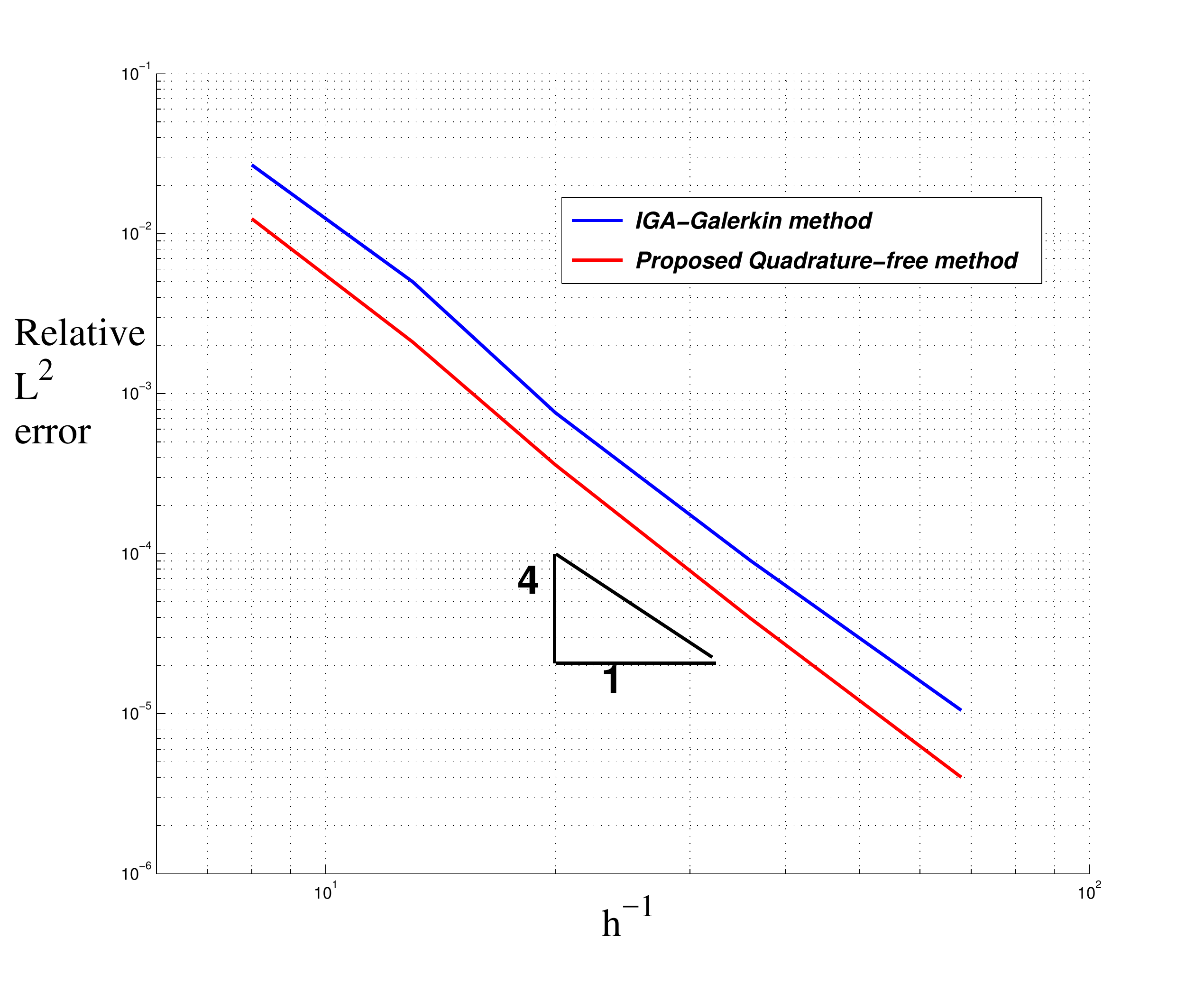}
\end{minipage}\\
\caption{Error history during the $h$-refinement process. }\label{fig:errorhistory}
\end{figure}

Obviously, if $G(u,v,w)$ is the exact representation of the rational
function $F(\ub)$, then we can achieve the exact solution of the
model problem by the proposed quadrature-free method. 
On the other hand, the accuracy of the proposed method depends on the 
degree of the \Bezier approximation $G(u,v,w)$, i.e,   $\alpha, \beta$
and $\gamma$. In our experimental results, we set $\alpha= 3l-3,
\beta=3m-3, \gamma=3n-3$ for the initial value. Similar with the
IGA-Galerkin method, there are three possible ways to improve the
approximation accuracy: (1) the
first possibility is to approximate $F(\ub)$ with piecewise \Bezier
polynomial with the same degree, which is similar with the
$h$-refinement by knot insertion in IGA-Galerkin method; (2) the second
way is to elevate the degree of the \Bezier approximation $G(u,v,w)$, 
which is similar with the $p$-refinement in IGA-Galerkin method; (3) the
third approach is to combine the piecewise method and the degree elevation 
method, which is similar with the $k$-refinement in IGA-Galerkin
method. 


In order to show the accuracy of the proposed quadrature-free method, 
some tests have been performed in our paper. In the presented numerical example, a
three-dimensional parameterization of 1/8th of a hollow sphere
is constructed with cubic B-spline volume as shown in Fig.\ref{fig:simulation}(a),  and the
source function $g({\bf x})$ in the model problem (\ref{eq:heat})  
with boundary condition
$T_0(\textbf{\emph{x}})=0$ is 
constructed such that the exact solution  (Fig.\ref{fig:simulation}(b))  is 
\begin{equation}
T(x,y,z)=\sin(x)\sin(y)\sin(z)(x^2+y^2+z^2-(R+r)^2) (x^2+y^2+z^2-(R-r)^2)
\end{equation}
in which $R=10$ and $r=1$. The simulating solution $T_h$ on the
isoparametric surface with $w=0.6$ of the proposed
quadrature-free method and the IGA-Galerkin method are illustrated in
Fig. \ref{fig:simulation}(c) and Fig.\ref{fig:simulation}(e), 
the corresponding error $(T-T_h)/|T|$ are plotted as
presented in Fig. \ref{fig:simulation}(d) and
Fig. \ref{fig:simulation}(f). Moreover, the relative $L^2$ error history during the
$h$-refinement by knot insertion for this numerical example is
presented in Fig. \ref{fig:errorhistory}. We can find that the
proposed quadrature-free method have the comparable accuracy with the
IGA-Galerkin method.  

\section{Computation-reuse for models with consistent volume parameterization}
\label{sec:example}

In this section, we present the details of our computation reuse framework for a set of
models with consistent volume parameterization based on the 
proposed quadrature-free IGA method. In the first phase, we
obtain consistent B-spline volumetric parameterization from given B-rep models 
and the template based domain. In the second phase, 
we perform analysis-reuse on the resulting consistent volumetric 
parameterization.  

\subsection{Framework overview}
\label{subsec:overview}

Generally, given a set of CAD models with consistent topology,  
our computation reuse framework can be described as following:

\noindent \textbf{Input}: a set of CAD models with consistent topology

\noindent \textbf{Output}: IGA results on all models  
\begin{enumerate}[{Step}~1:]
\item construct topology-consistent volumetric parameterization for
  the input set of CAD models and the template base domain as
  described in Section \ref{sec:volcon}.
  
\item perform \Bezier extraction for the B-spline based volume parameterization of
  one CAD model, then the conversion matrix can be stored and reused for 
  the \Bezier extraction of volume parameterization of other CAD models.       
  
\item for a specified PDE problem (i.e., heat conduction problem), impose the boundary condition by 
boundary collocation method described in subsection \ref{sec:BCimposition}. 

\item solve the specified PDE problem on the CAD model in Step 3 with
  the quadrature-free IGA method proposed in Section \ref{sec:analysis}. 
  
\item solve the specified PDE problem for the other CAD models, in
  which the boundary collocation matrix can be reused and the entry evaluation of 
  stiffness matrix in Step 4 can be partially reused. 
  
\item output the results of IGA.
\end{enumerate}


Based on this framework, a significant improvement on the efficiency can be achieved for a set of models with consistent-topology volume parameterizations. In the following subsections, we will discuss the computation reuse mechanism during the entry evaluation of stiffness matrix and the boundary condition imposition.

\subsection{Computational reuse for entries of stiffness matrix}

In our computation resue framework, there are mainly two reuse parts: 
the entry evaluation of element stiffness matrix and the boundary condition 
imposition.

\subsubsection{Entry evaluation of element stiffness matrix} 

After the heat conduction problem is solved on one of the input models, we want to reuse 
some computation for the IGA solving on the other models. As describe in subsection \ref{subsection:stiffness}, there are mainly three computing parts for the entry evaluation of element 
stiffness matrix. 

\noindent 1) The first part is for the computation of $J_{ijk}$ in Proposition
\ref{propostion:stiffness}. From the formula of $J_{ijk}$ as shown 
in Eq.(\ref{jacobian}), the values of $D_{ijk}$ in Eq.(\ref{Dijk}) can be stored and 
reused for the following models. Therefore, we only need to compute 
the coefficients depends on the control points for different \Bezier solids with the same 
degree. 

\noindent 2) The second part is for the computation of $F_{ijk}$ in  Proposition
\ref{propostion:stiffness}. As shown in the proof of Proposition 4.5, the computational of 
$F_{ijk}$ can be considered as a repeated process of production on two trivariate \Bezier functions. 
In summary, we have to apply the production for totally $54$ times to obtain the value of 
$F_{ijk}$. As the set of given B-spline models have the same basis
functions with the same degree, the coefficients which are similar as in Eq.(\ref{Dijk}) can be stored
and reused to compute the $F_{ijk}$ for the other volumetric models 
in a similar way as the computational reuse for $J_{ijk}$.   
 
\noindent 3) The third part is for the \Bezier approximation of the rational
trivariate \Bezier function $F(\ub)$. As shown in Eq.(\ref{solution}),
the \Bezier approximation  $G(u,v,w)$ of $F(\ub)$ can be obtained as 
\begin{equation}\label{solutionreuse}
{\bf G}=\sigma \cdot {\bf E}^{-1}\cdot {\bf L}^{-1} {\bf Q}
\cdot {\bf F}. 
\end{equation}
We can find that  in Eq.(\ref{solutionreuse}), given 
the degree $(\alpha,\beta, \gamma)$ of \Bezier approximation and the
degree $(l,m,n)$ of 
the B-spline model, $\sigma, {\bf L}^{-1}$ and $ {\bf Q}$ are
independent with the control points information. As a result, after 
solving the IGA problem on the first model, $\sigma, {\bf L}^{-1}$ and
${\bf Q}$ can be stored and reused for the following IGA solving
process on  all other models.

\noindent Overall, given a set of volumetric models with the
same B-spline basis function representation (i.e., the same degree and the
same knot vectors for each block), the computation can be partially reused 
after solving the IGA problem on one model. 
For the new IGA solving process, the parts that we need compute for element stiffness matrix filling  are
the following formula involved in the computation of $J_{ijk}$,
\begin{equation}
\mathbf{det} \left({\begin{array}{c}
      \pb_{i_1+1,j_1,k_1}-\pb_{i_1,j_1,k_1}\\
      \pb_{i_2,j_2+1,k_2}-\pb_{i_2,j_2,k_2}  \\
      \pb_{i_3,j_3,k_3+1}-\pb_{i_3,j_3,k_3}) \\
 \end{array}} \right)^T,
\end{equation}
and the vector $\mathbf{F}=[F_{pqr}]$   with $F_{pqr}$ as entries defined
in Eq.(\ref{eqn:F(u)}). 

After the local stiffness matrix for each element are filled, the
global stiffness matrix can be obtained by assembling. In the assembly
process, the boundary condition described in the governing equation
must be imposed and the corresponding entries related to the boundary
condition will also be evaluated. In the following subsection, we will discuss the reuse for the imposition of boundary conditions.

\subsubsection{Imposition of boundary conditions}
\label{sec:BCimposition}

As the Bernstein basis functions do not have interpolating property
at control points,  we cannot 
impose the essential boundary conditions directly onto the   control
variables on the boundary. Special treatments need to be implemented
to achieve the specified boundary conditions, such as the least square
approach, the penalty function method or the Nitsche method. In this
paper, a collocation method is employed to impose boundary
conditions. For a set of computational models with topology-consistent
volume parameterization, the control variables on the boundary can be reused for the same boundary conditions, and the collocation matrix can be reused for different boundary conditions.   

Suppose that $\{{\bf x}_i\}_{i=0}^{n_b}$ are collocation points on the
boundary surface and  $\{\xi_i,\eta_i,\zeta_i \}_{i=0}^{n_b}$ are
corresponding parametric coordinates in the parametric domain, 
the Dirichlet boundary condition $U\|_{\partial \Omega}=h(\{\bf x \})$
can be defined as 
\begin{equation}\label{eqn:boundary}
\sum_{j=1}^k N_j(\xi_i, \eta_i,\zeta_i)b_j =h(\{{\bf x}_i\}), \qquad
\qquad  i=0, \cdots, n_b
\end{equation}
in which $h(\{{\bf x}_i\})$ are boundary values to be interpolated,
and $b_j$ are control variables to be solved. This equation can be rewritten into a matrix form as 
\[
{\bf M} {\bf B}= {\bf H}, 
\] 
in which the entry of matrix ${\bf M}$ is $M_{ij}=N_j(\xi_i,
\eta_i,\zeta_i)$, $\bf H=[h_i]=[h(\{{\bf x}_i\})]$, $i,j=0, \cdots, n_b$ . Then the 
boundary control variables $b_j$ can be solved by 
\begin{equation}\label{imposition}
{\bf B}= {\bf M}^{-1}{\bf H}. 
\end{equation}

In this paper, since the models with topology-consistent 
volume parameterization have the same basis functions for the corresponding blocks, 
the boundary control variables which is determined by Eq.(\ref{imposition}) can be used for all other models for a PDE problem with the same boundary conditions.
Furthermore, for the PDE problem with different boundary conditions, 
the inverse of collocation matrix ${\bf M}^{-1}$ can also be reused for the 
models with topology-consistent volume parameterization. In practice,
we apply LU-decomposition on the sparse matrix ${\bf M}$, and re-use the
decomposition in Eq.(\ref{imposition}) to determine the value of ${\bf B}$.

\subsection{Experimental results}

In this subsection,   experimental results will be presented to show the advantage of the proposed computation reuse method. 

Three sets of CAD models are tested in this paper. The first set has
two hand models (Fig. \ref{fig:hand}), four airplane models are tested
in the second set (Fig. \ref{fig:airplane}), and the third set
consists of four human models with consistent topology
(Fig. \ref{fig:human}). For each model in Figs. \ref{fig:hand}-\ref{fig:human},
from left to right we show the input surface model with patch-partition information, the results of boundary B-spline surface fitting, the results of volume parameterization, and the IGA results for a heat conduction problem respectively.
 
In order to illustrate the effectiveness of the proposed computation reuse approach, the corresponding average assembling time of the IGA and IGA-reuse approach are shown in Table \ref{table:data}
for the models in Figs. \ref{fig:hand}-\ref{fig:human} respectively. 
The corresponding extra storages of the proposed IGA-reuse method are
also given. 
All the computations are implemented in C++ and timed on a Macbook Pro with a 
quad core 2.4 GHz Intel Core i7 processor and 8GB RAM.  From the performance statistic shown in the table, we can find that the computational costs in assembling stiffness matrix can be reduced
significantly (i.e, up to 12 times) by our method while keeping the nice trade-off between
efficiency and storage. Furthermore, as shown in Fig. \ref{fig:datacurve}, the
acceleration ratio keeps increasing while increasing the degree of freedom in IGA.

Overall, since the evaluation of high-order basis function for one model can be reused for other models with
consistent topology, a similar performance of classical linear finite element method can be achieved for isogeometric analysis on a set of models. This addresses a main shortcoming of isogeometric analysis and makes it cost-efficient in solving large-scale problems in computational mechanics. 

\begin{table}  
\caption{Quantitative data and average assembling time for computation reuse in IGA}

\begin{tabular}{|l|l|l|l|l|l|l|l|} 
    \hline 
    Example&$p$& $h$ & \#DOF   & IGA  & IGA-reuse & 
    acceleration ratio & extra storage(NNZ) \\ [1.0ex] 
    \hline 
    \hline 
   \multirow{3}{3em}{Hand (Fig.\ref{fig:hand})} & \multirow{3}{3em}{2}
     & $h=1$&85,169   & 3.93 sec.  & 1.62 sec. & $\times$2.42 & 1,245,194 \\[1.0ex]
    && $h=2$&562,952   & 19.83 sec.  & 5.43 sec. & $\times$3.65 & 12,012,034 \\[1.0ex]
    && $h=3$&4,240,664   & 136.14 sec.  & 15.81 sec. & $\times$8.61& 68,023,002 \\[1.0ex]
\hline    
  \multirow{3}{3em}{Airplane (Fig.\ref{fig:airplane})} & \multirow{3}{3em}{3}
    & $h=1$&35,154  & 3.48 sec. & 1.14 sec. & $\times$3.05 &  262,202\\[1.0ex]
    && $h=2$&263,424  & 14.38 sec.  & 3.53 sec. & $\times$4.07 & 7,002,134\\[1.0ex]
    && $h=3$&1,843,968  & 93.22 sec. & 10.15 sec.  & $\times$9.18& 35,082,344\\[1.0ex]
\hline   
 \multirow{3}{2em}{Human (Fig.\ref{fig:human})} &
   \multirow{2}{3em}{4} & $h=1$  &153,472 & 13.97 sec.  & 2.71 sec. &
   $\times$5.15 & 4,044,174\\[1.0ex]
 & & $h=2$&1,185,408  & 87.45 sec. & 9.24 sec. &   $\times$9.46 &23,432,642\\[1.0ex]
& & $h=3$&7,938,056  & 584.82 sec.  & 46.23 sec. &  $\times$12.65& 165,206,720\\[1.0ex]
\hline 
\end{tabular}
\\
\noindent \#DOF: the degree of freedom ;  \quad $p$: the degree of
basis function; $h$: $h$-refinement step; \\ NNZ: number of none-zero
elements in matrix.  
  \label{table:data} 
\end{table} 

\begin{figure*}
\centering
\begin{minipage}[t]{6.0in}
\centering
\includegraphics[width=3.4in,height=1.3in]{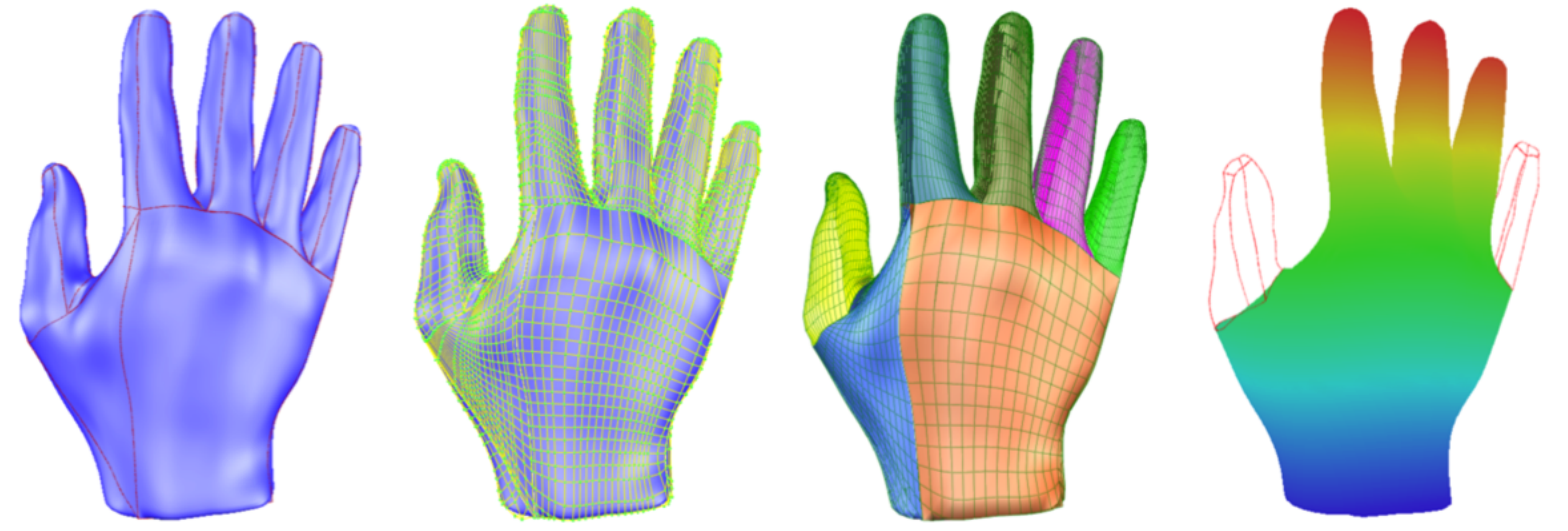}
\\ (a) Hand model I
\end{minipage}\\
\begin{minipage}[t]{6.0in}
\centering
\includegraphics[width=3.4in,height=1.55in]{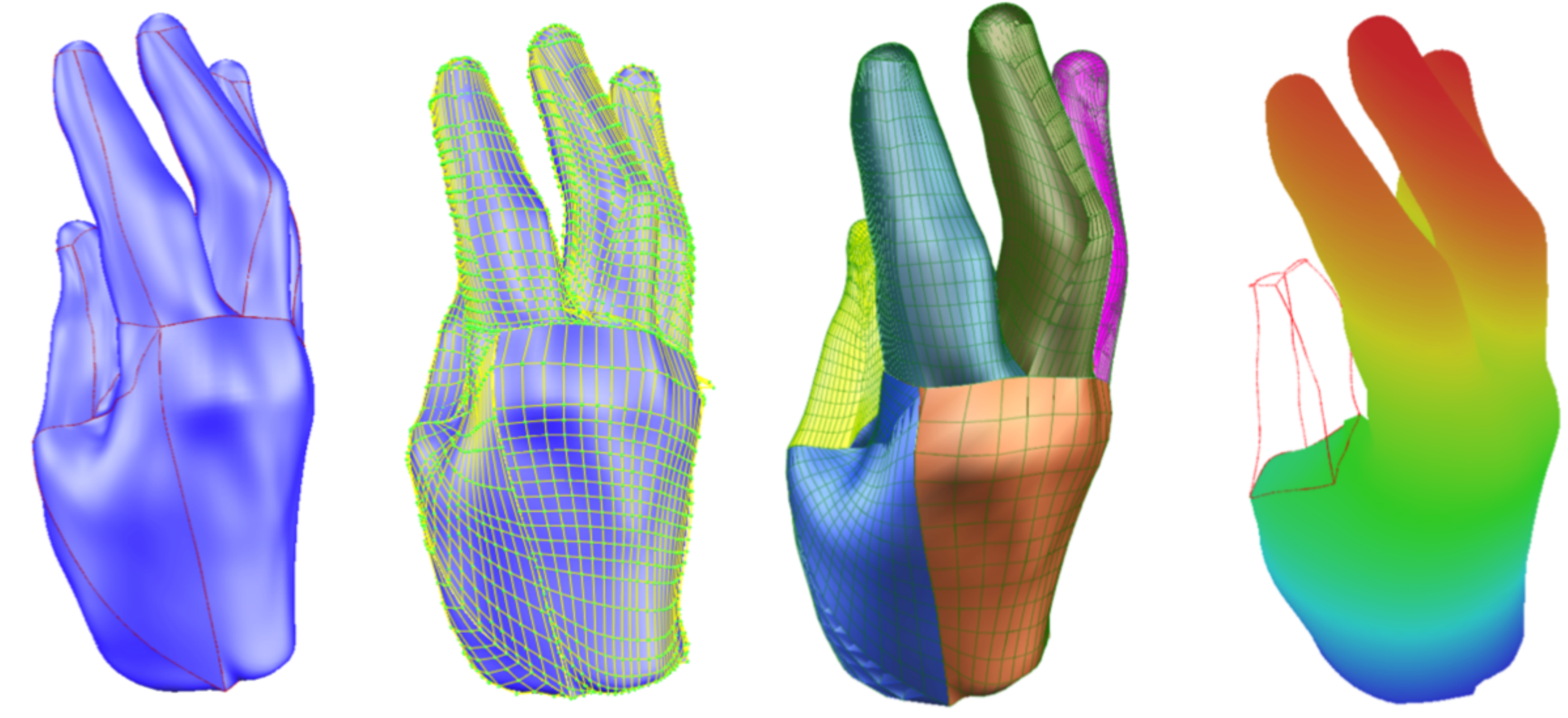}
\\ (b) Hand model II
\end{minipage}\\
\caption{Computation reuse in IGA on human body models with consistent volumetric parameterization.}
\label{fig:hand}
\end{figure*}

\begin{figure*}
\centering
\begin{minipage}[t]{6.0in}
\centering
\includegraphics[width=6.5in]{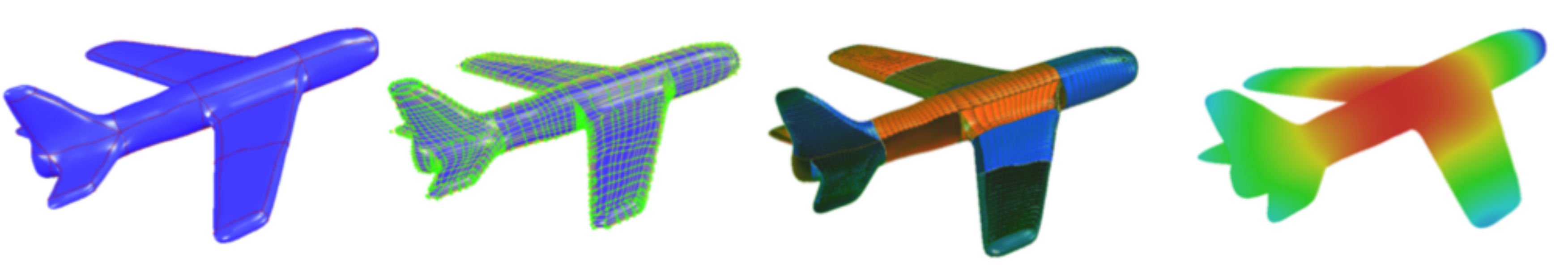}
\\ (a) Airplane model I
\end{minipage}\\
\begin{minipage}[t]{6.0in}
\centering
\includegraphics[width=6.5in]{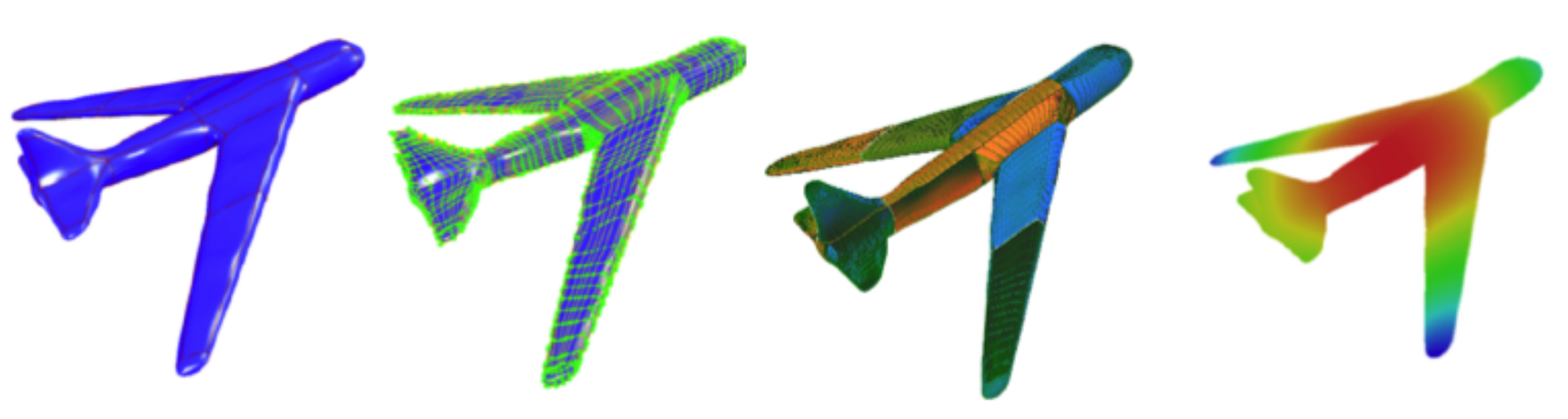}
\\ (b) Airplane model II
\end{minipage}\\
\begin{minipage}[t]{6.0in}
\centering
\includegraphics[width=6.5in]{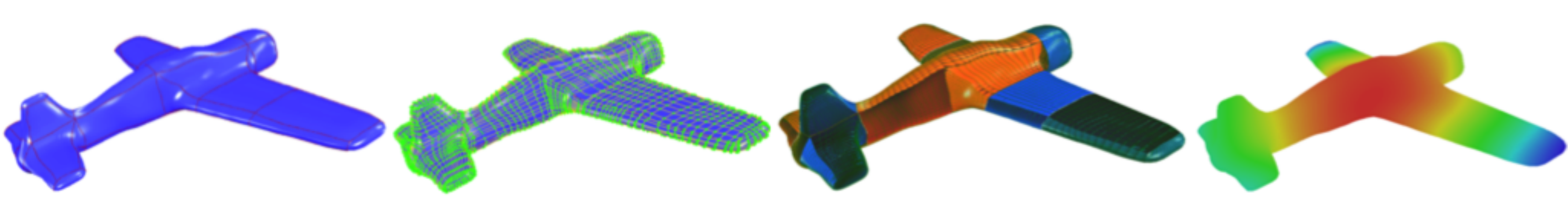}
\\ (c) Airplane model III
\end{minipage}\\
\begin{minipage}[t]{6.0in}
\centering
\includegraphics[width=6.5in]{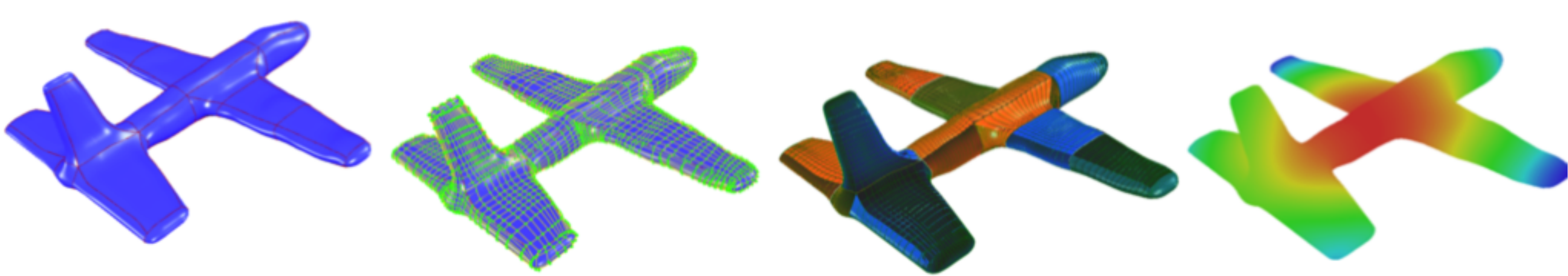}
\\ (d) Airplane model IV
\end{minipage}\\
\caption{Computation reuse in IGA  on airplane models with consistent volumetric parameterization. }
\label{fig:airplane}
\end{figure*} 

\begin{figure*}
\centering
\begin{minipage}[t]{6.0in}
\centering
\includegraphics[width=4.5in]{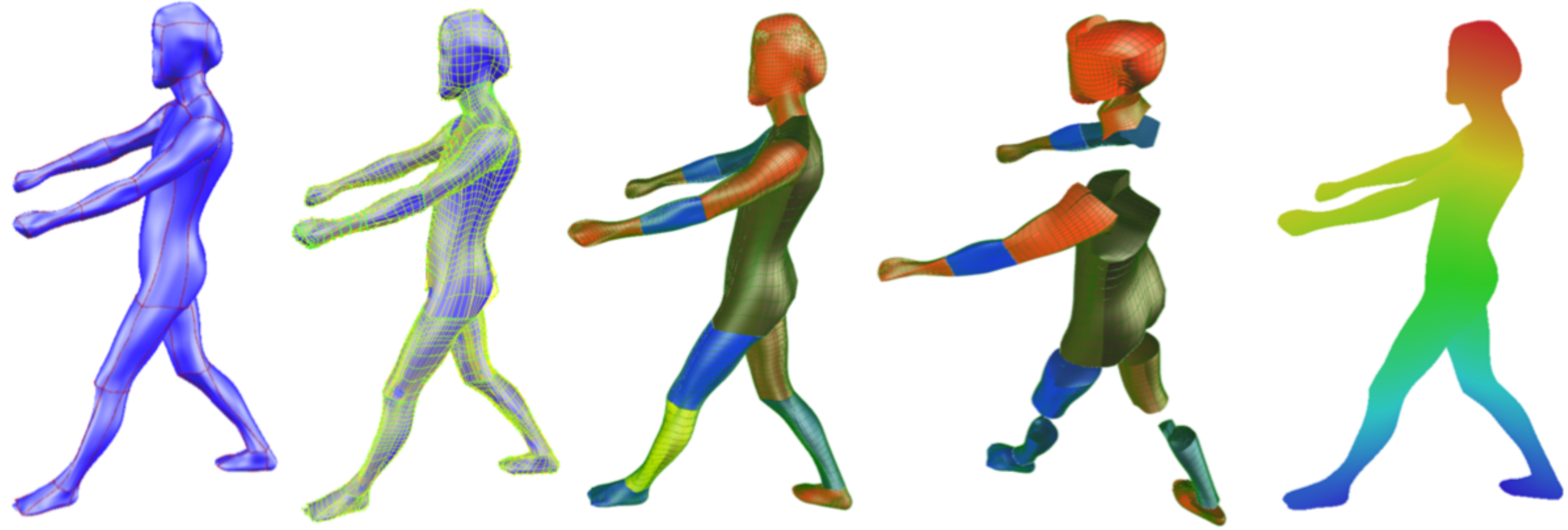}
\\ (a) Human body model I
\end{minipage}\\
\begin{minipage}[t]{6.0in}
\centering
\includegraphics[width=4.5in]{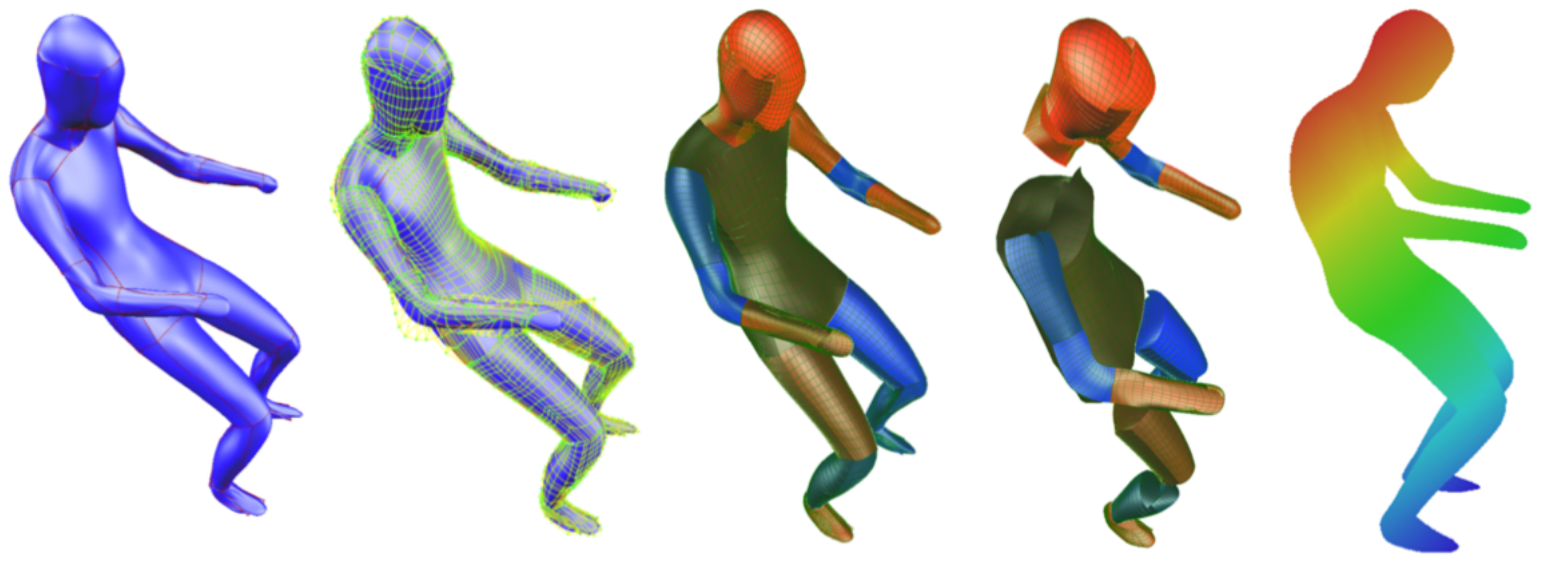}
\\ (b) Human body model II
\end{minipage}\\
\begin{minipage}[t]{6.0in}
\centering
\includegraphics[width=4.5in,height=2.in]{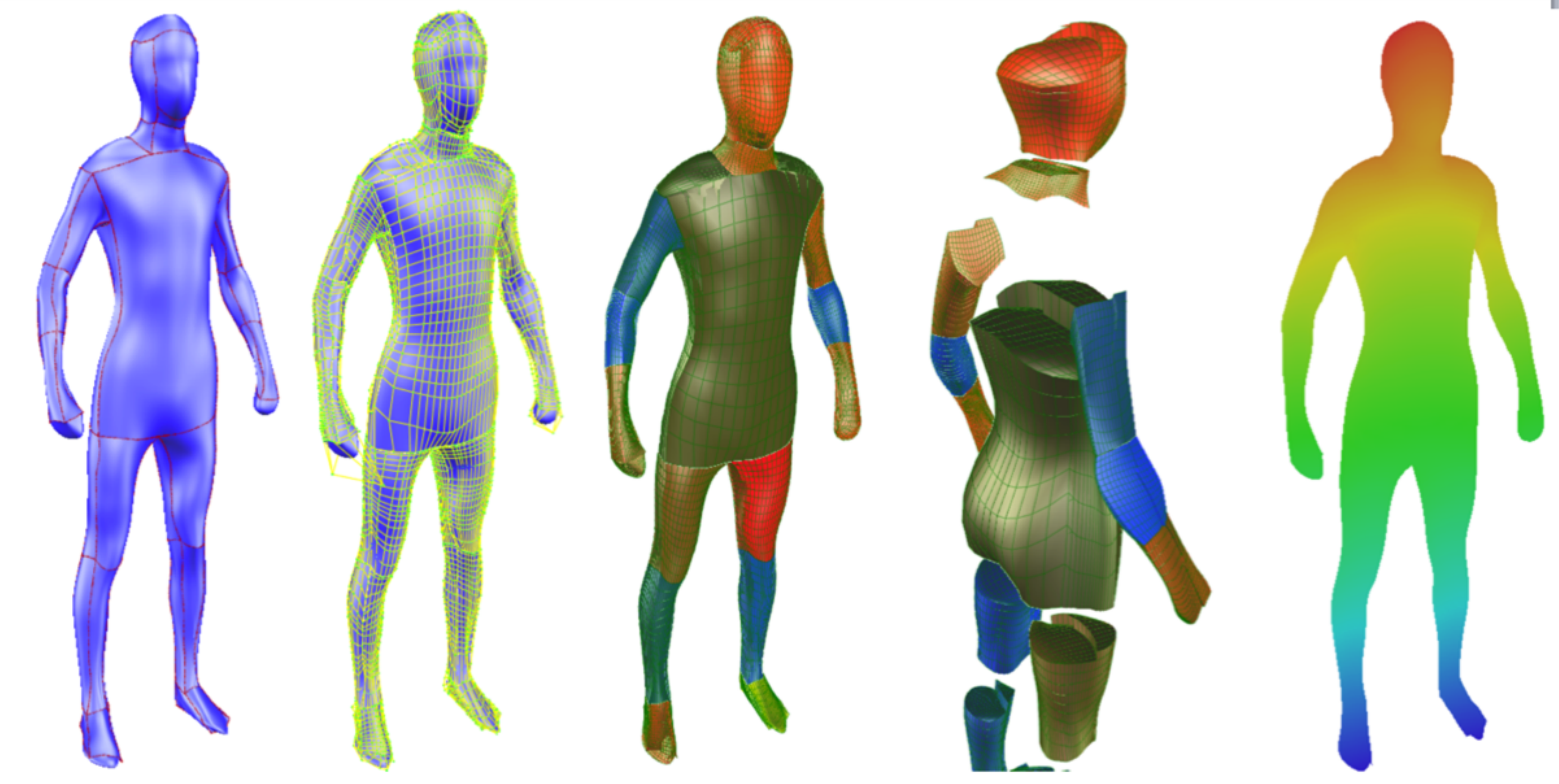}
\\ (c)Human body model III
\end{minipage}\\
\begin{minipage}[t]{6.0in}
\centering
\includegraphics[width=4.5in,height=1.8in]{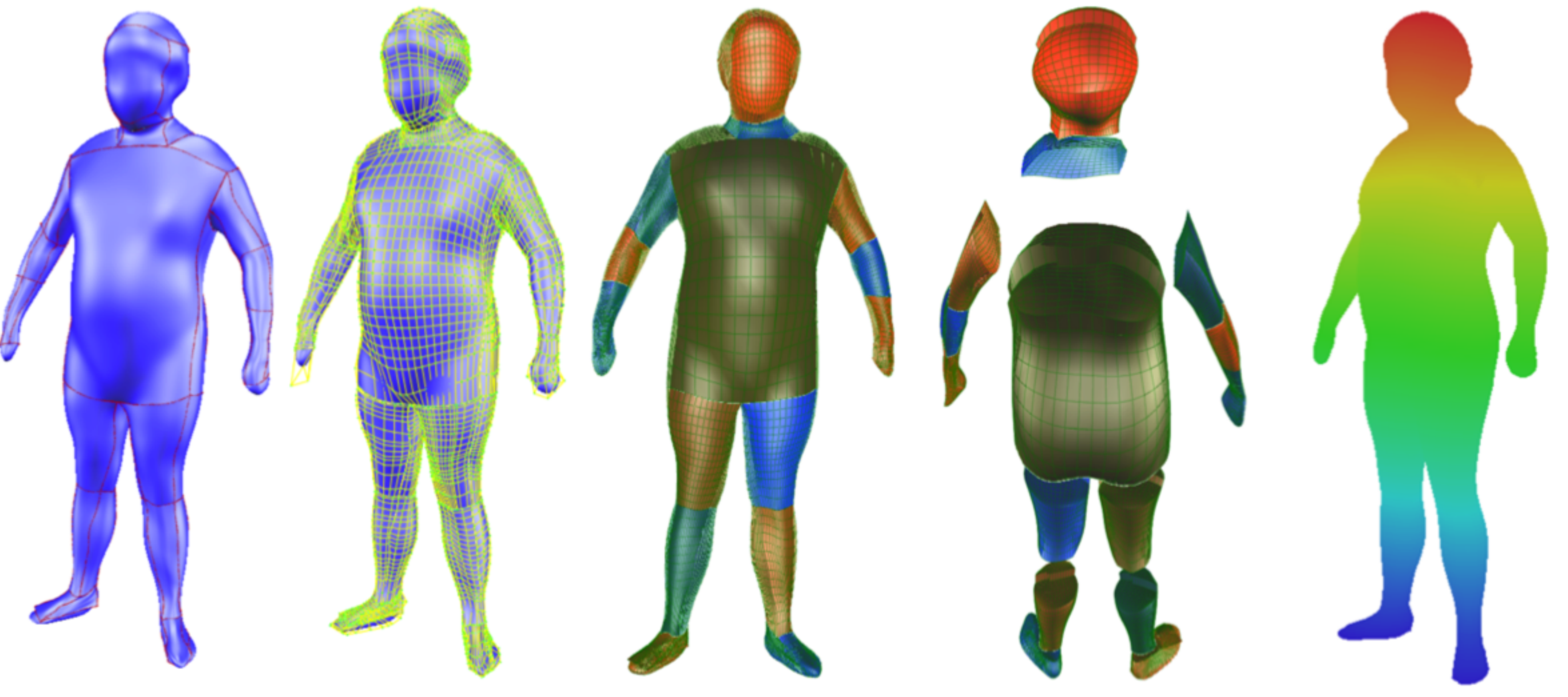}
\\ (d) Human body model IV
\end{minipage}\\
\caption{Computation reuse in IGA on human body models with consistent volumetric parameterization. }
\label{fig:human}
\end{figure*}

\begin{figure*}[t]
\centering
\begin{minipage}[t]{2.0in}
\centering
\includegraphics[width=2.1in]{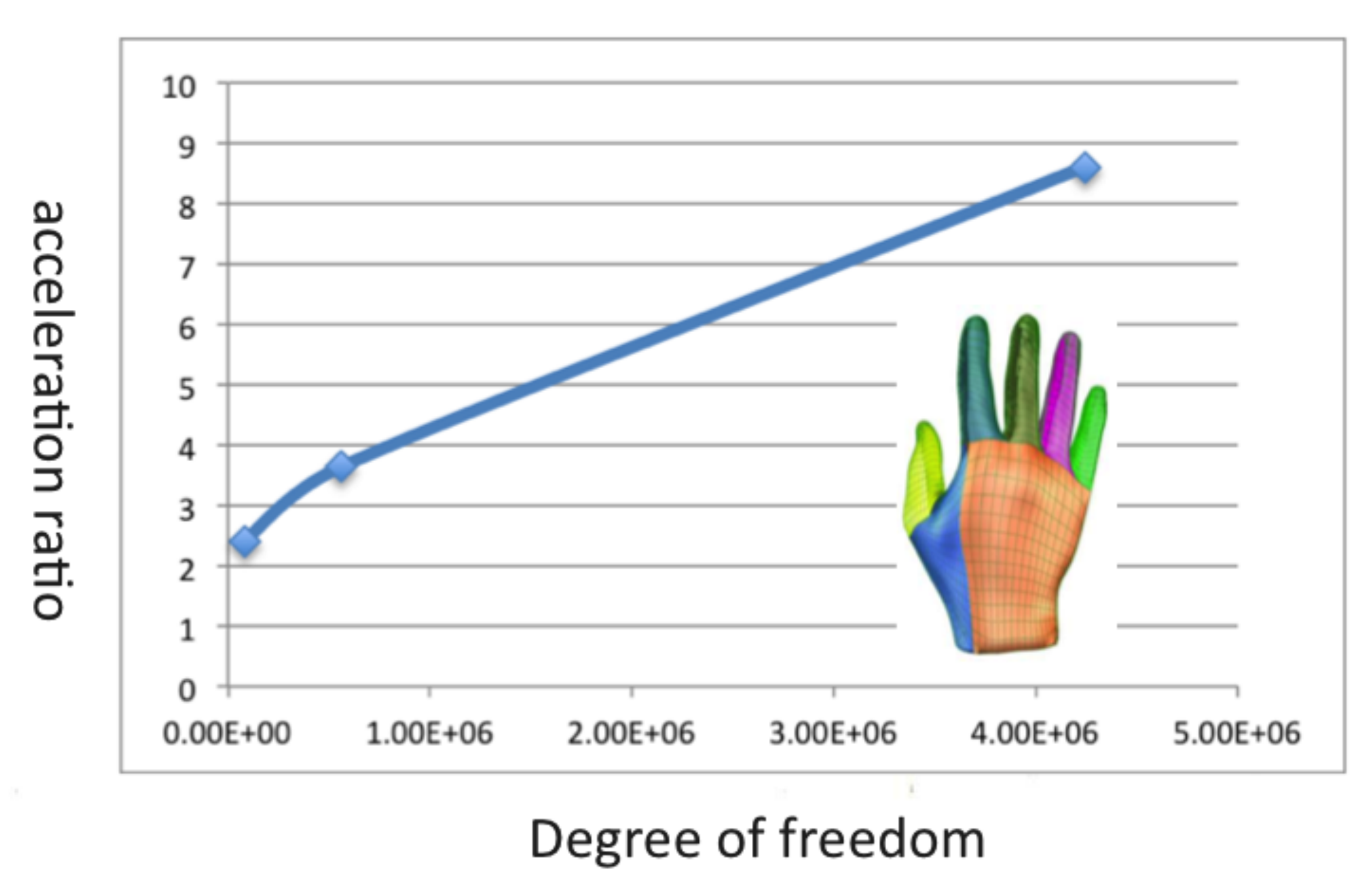}
\\ (a) hand models
\end{minipage}
\begin{minipage}[t]{2.0in}
\centering
\includegraphics[width=2.1in]{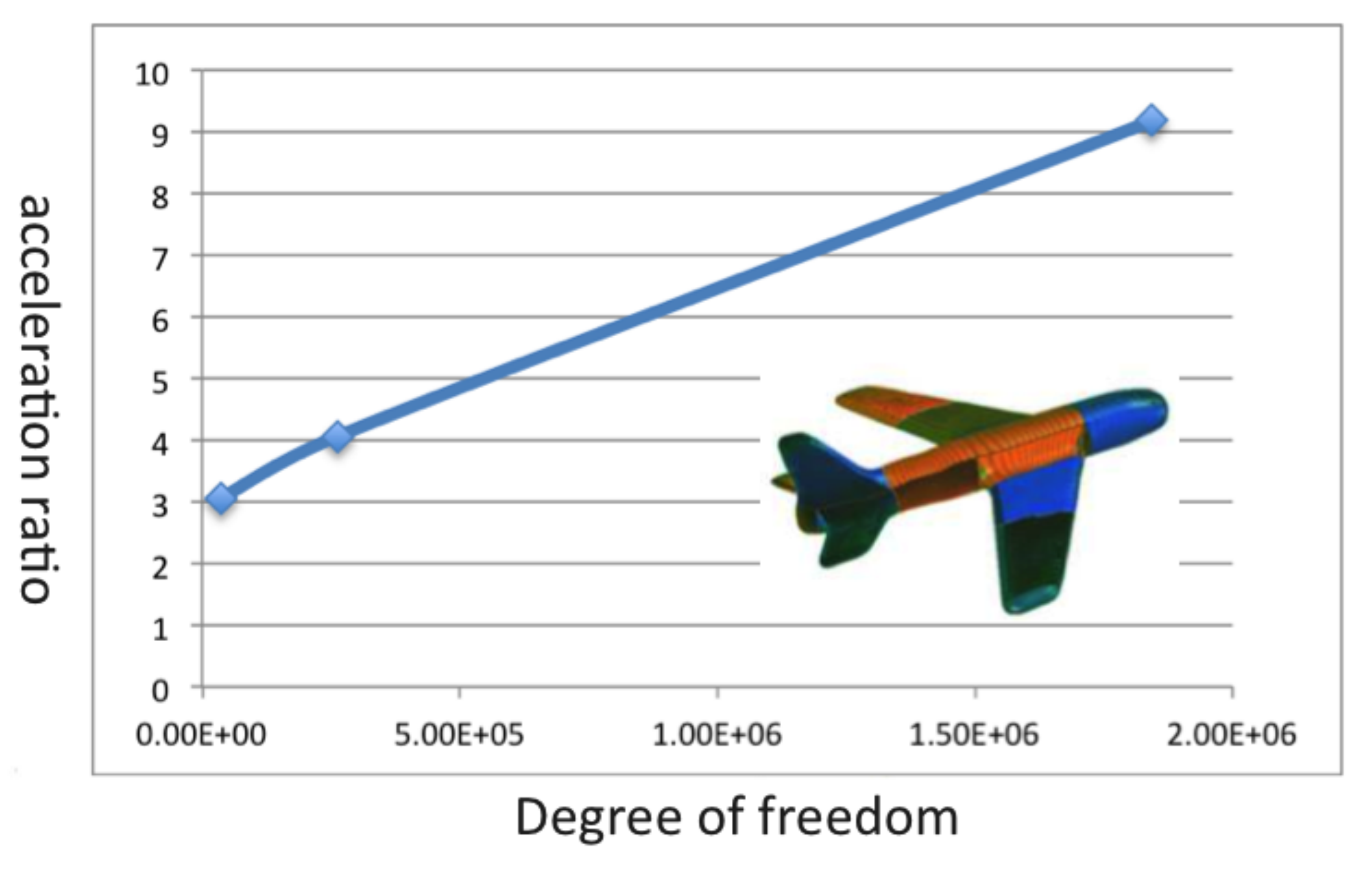}
\\ (b) airplane models
\end{minipage}
\begin{minipage}[t]{2.1in}
\centering
\includegraphics[width=2.1in]{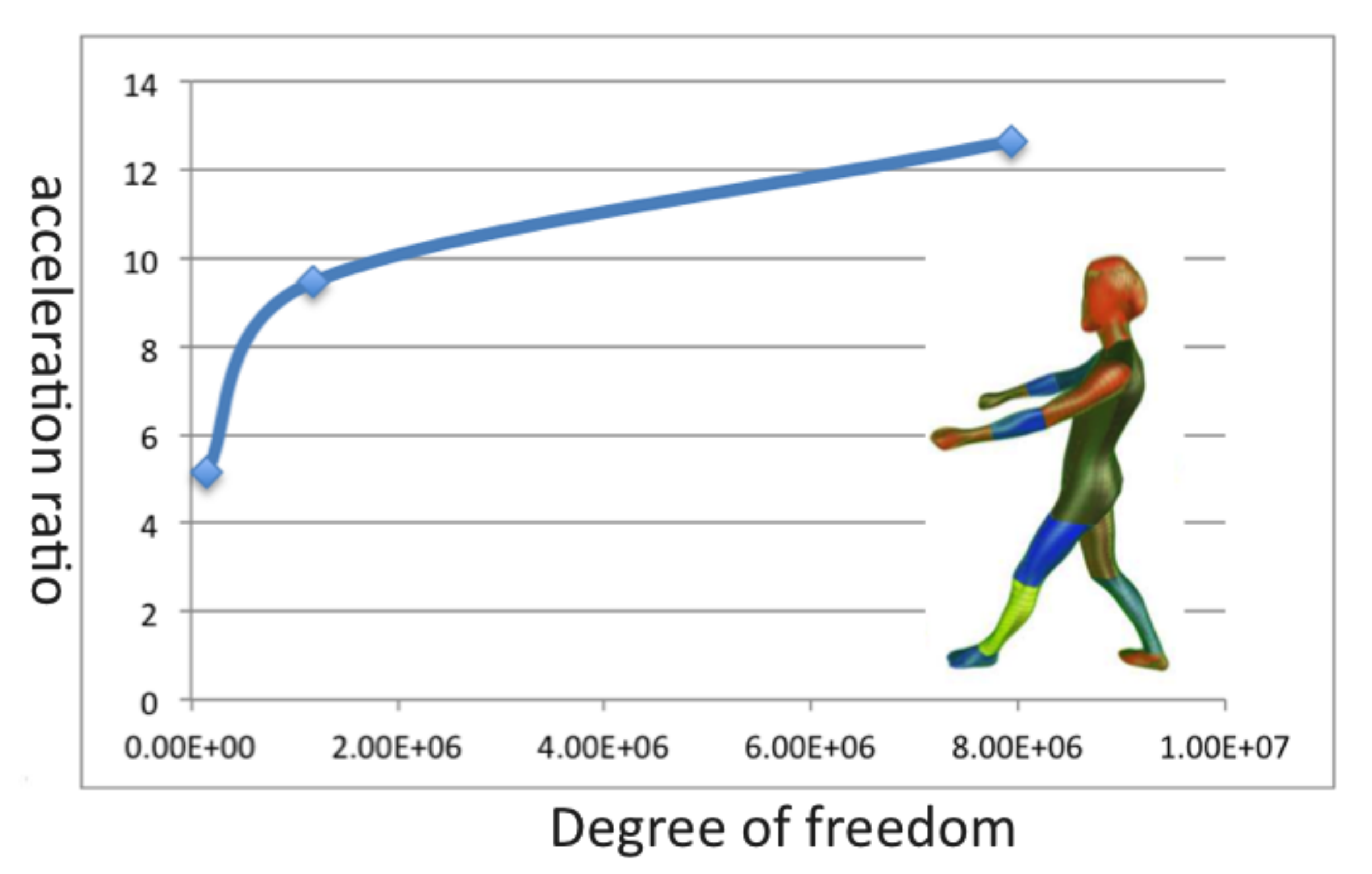}
\\ (c) human models
\end{minipage}\\
\caption{Acceleration ratio and the degree of freedom for examples in Figs.\ref{fig:hand}-\ref{fig:human}. The horizontal axis denotes the degree of freedom, and the vertical axis illustrates the acceleration ratio.}
\label{fig:datacurve}
\end{figure*}

\section{Conclusion}\label{sec:conclude}

In this paper, in order to improve the computational efficiency of 
isogeometric analysis, the concept of \emph{computation reuse} is proposed for 
three-dimensional models with similar semantic features.  
For a set of models with consistent topology, a CSRBF-based 
method is firstly applied to construct topology-consistent
volumetric B-spline parameterization from given template domains. 
After obtaining the consistent volumetric parameterization,  we propose an efficient
quadrature-free method to compute the entries of stiffness
matrix with the techniques of \Bezier extraction and polynomial approximation.
With the help of our method, evaluation on the stiffness matrix and
imposition of the boundary conditions can 
be pre-computed and partially reused for models having consistent
volumetric parametrization. Effectiveness of the proposed methods has been verified on
several examples with complex geometry, and the computation efficiency similar to classical 
finite element method can be achieved for IGA on these models.

We plan to test the proposed computation reuse approach on other
physical simulation problems in the future.  Problems such as linear elasticity and flow simulation for a set of models with consistent topology have been
widely practiced in the design of a family of products, which will
potentially be benefited from our approach.  On the other aspect, this
technique can be used as the inner loop of physics-based shape
optimization , in which the computation can be significantly speeded
up after the first isogeometric analysis on the original shape.

 \section*{Acknowledgment}
Gang Xu is partially supported by the National Nature Science Foundation of China (No.
61472111) and Zhejiang Provincial Natural Science Foundation of China
(No. LR16F020003).

\bibliographystyle{abbv}

\end{document}